\documentclass[12pt,reqno ]{amsart}
\usepackage{amsmath}
\usepackage{amssymb}
\usepackage{hyperref}
\usepackage{amsfonts}

\theoremstyle{proclaim}
\newtheorem{theorem}{Theorem}[section]
\newtheorem{lemma}[theorem]{Lemma}
\newtheorem{corollary}[theorem]{Corollary}
\newtheorem{proposition}[theorem]{Proposition}
\theoremstyle{statement}
\newtheorem{remark}[theorem]{Remark}
\numberwithin{equation}{section}

\begin{document}

\title[Invertible Toeplitz Products]{Invertible Toeplitz products, weighted norm inequalities, and A${}_p$ weights}

\author[J. Isralowitz]{Joshua Isralowitz}
\address{Mathematisches Institut \\
Georg-August Universit\"{a}t  G\"{o}ttingen \\
Bunsenstra$\ss$e 3-5 \\
D-37073, G\"{o}ttingen \\
Germany}
\email{	jbi2@uni-math.gwdg.de}
\thanks{The first author was supported by an Emmy-Noether grant of Deutsche Forschungsgemeinschaft}

\begin{abstract}
In this paper, we characterize invertible Toeplitz products on a number of Banach spaces of analytic functions, including the weighted Bergman space $L^p _a (\mathbb{B}_n, dv_\gamma)$, the Hardy space $H^p(\partial \mathbb{D})$, and the standard weighted Fock space F${}_\alpha ^p$ for $p > 1$.  The common tool in the proofs of our characterizations will be the theory of weighted norm inequalities and A${}_p$ type weights. Furthermore, we prove weighted norm inequalities for the Fock projection, and compare the various A${}_p$ type conditions that arise in our results. Finally, we extend the  ``reverse H\"older inequality" of Zheng and Stroethoff \cite{SZ1, SZ2} for $p = 2$ to the general case of $p > 1$.\end{abstract}



\maketitle

\section{Introduction}
Let $\mathbb{B}_n$ denote the unit ball in $\mathbb{C}^n$ and let $dv$ denote the usual normalized volume measure on $\mathbb{B}_n$.  For $\gamma  > -1$, let $dv_\gamma (z) = c_\gamma (1 - |z|^2)^\gamma dv(z)$ where $c_\gamma$ is a normalizing constant.  For $1 \leq p < \infty$, the Bergman space $L^p _a (\mathbb{B}_n, dv_\gamma)$ is the Banach space of analytic functions on $\mathbb{B}_n$ that belong to $L^p  (\mathbb{B}_n, dv_\gamma)$.

\medskip
As a (formal) limiting case $\gamma \rightarrow -1^+$ of the spaces $L_a ^p  (\mathbb{B}_n, dv_\gamma)$, one obtains the Hardy space $H ^p  (\partial \mathbb{B}_n)$, which is the closure in $L^p(\partial \mathbb{B}_n, d\sigma)$ of analytic polynomials on $\partial \mathbb{B}_n$ where $d\sigma$ is the standard surface measure on $\partial \mathbb{B}_n$ (more precisely, $dv_\gamma \stackrel{\text{wk}^*}{\longrightarrow} d\sigma$ on $C(\overline{\mathbb{B}_n})$ as $\gamma \rightarrow -1^{+}$.)  As another (formal) limiting case where $\gamma \rightarrow +\infty$, one obtains the Fock space $F_\alpha ^p$ of all entire functions $f$ where $f(\cdot) e^{- \frac{\alpha}{2} |\cdot|^2}$ is in $L^p(\mathbb{C}^n, \left(p \alpha/ 2 \pi\right)^n dv)$ for $\alpha  > 0$ (and where $F_\alpha ^p$ is equipped with its canonical Banach space norm.)

\medskip
It is well known \cite{Z2} that the orthogonal projection $P_\gamma$ from $L^2  (\mathbb{B}_n, dv_\gamma)$ onto $L_a ^2  (\mathbb{B}_n, dv_\gamma)$ is given by \begin{align} P_\gamma f (z) = \int_{\mathbb{B}_n} K_\gamma (z, u) f(u) dv_\gamma(u) \nonumber \end{align} where $K_\gamma(z, u)$ is the Bergman kernel $K_\gamma (z, u) = (1 - z \cdot u)^{-(n + 1 + \gamma)}$. Let  $p > 1$  and let $q$ be the conjugate exponent of $p$.  If $g \in L^q (\mathbb{B}_n, dv_\gamma)$, then we can define the Toeplitz operator $T_g$ on $L_a ^p  (\mathbb{B}_n, dv_\gamma)$ by the formula $T_g = P_\gamma M_g$ (with $M_g$ being ``multiplication by $g$"). Similarly, if $g \in L^q(\partial \mathbb{B}_n)$, then the Toeplitz operator $T_g$ is defined on $H ^p  (\partial \mathbb{B}_n)$ by $T_g = P^+ M_g$ where $P^+$ is the Hardy projection. Note that while $T_g = P_\gamma M_g$ obviously depends on $\gamma$, for the sake of notational ease we will still refer to this Toeplitz operator on $L_a ^p(\mathbb{B}_n, dv_\gamma)$ by $T_g$.  The same will be true when we define Toeplitz operators on Fock spaces $F_\alpha ^p$ in Section $3$.

 \medskip   Toeplitz operators $T_g$ on both the Hardy space and the Bergman space have been extensively studied in the literature when $p = 2$. (see \cite{Z1} and the references therein.)  However,  it is well known \cite{Z2} that both the Bergman projection $P_\gamma$ and the Hardy projection $P^+$ are bounded on  $L ^p  (\mathbb{B}_n, dv_\gamma)$ and $L^p (\partial \mathbb{B}_n, d\sigma)$, respectively, whenever $p > 1$.  Thus, many of the results regarding Toeplitz operators for $p = 2$ can be appropriately generalized to the $p > 1$ case.

 \medskip
 In \cite{SZ1, SZ2}, the invertibility of the product of Toeplitz operators $T_f T_{\overline{g}}$ for analytic $f$ and $g$ was characterized for the Bergman space $L_a ^2 (\mathbb{B}_n, dv_\gamma)$ and the Hardy space $H^2(\partial \mathbb{B}_n)$ when $n = 1$. In particular, they proved the following result (where dA${}_\gamma$ is the weighted area measure on the unit disk $\mathbb{D}$) :

\begin{theorem}  For functions $f, g \in H^2 (\partial \mathbb{D})$, the Toeplitz product $T_f T_{\overline{g}}$ is bounded and invertible on $H^2 (\partial \mathbb{D})$ if and only if \begin{align} \underset{u \in \mathbb{D}}{\inf}  \ |f (u)||g(u)| > 0 \nonumber \end{align}  and   \begin{align}  \underset{u \in \mathbb{D}}{\sup} \  \widehat{|f|^2} (u) \widehat{|g|^2}(u)  < \infty. \nonumber  \end{align}  Moreover, for $f, g \in L_a ^2 (\mathbb{D}, dA_\gamma )$, $T_f T_{\overline{g}}$ is bounded and invertible on $L_a ^2 (\mathbb{D}, dA_\gamma )$ if and only if \begin{align} \label{1.1} \underset{u \in \mathbb{D}}{\inf}  \  |f (u)||g(u)| > 0 \tag{1.1}  \end{align} and  \begin{align}  \underset{u \in \mathbb{D}}{\sup}  \ B_\gamma (|f|^2) (u) B_\gamma (|g|^2)(u)  < \infty.     \tag{1.2}  \end{align} \end{theorem}

 \noindent Here, $\widehat{f}$ is the Poisson extension of a function $f$ on $\partial \mathbb{D}$ and $B_\gamma{f}$ is the Berezin transform of a function $f$ on $\mathbb{D}$ given by \begin{align} B_\gamma (f) (z) = \int_{\mathbb{D}} f(u) |k_z ^\gamma (u)|^2 dA_\gamma (u) \nonumber \end{align} where $k_z ^\gamma$ is the normalized Bergman kernel $k_z  ^\gamma (u) = K_\gamma (u, z)/ \sqrt{K_\gamma (z, z)}$ of $L_a ^2 (\mathbb{D}, dA_\gamma)$. For the sake of notational ease, we will drop the $\gamma$ in the notation for $ k_z ^\gamma$ in the rest of the paper.

\medskip
The main step in proving Theorem $1.1$ (in both the Bergman and Hardy space settings) is showing that the hypotheses in Theorem $1.1$ are enough to guarantee the boundedness of $T_f T_{\overline{g}}$.  Once this is done, an easy argument from \cite{SZ1, SZ2} completes the proof.

     \medskip
     To prove the boundedness of $T_f T_{\overline{g}}$, the authors first proved in \cite{SZ2} that for $f, g \in L_a ^2 (\mathbb{D}, dA_\gamma)$, we have that $T_f T_{\overline{g}}$ is bounded on  $L_a ^2 (\mathbb{D}, dA_\gamma)$ if there exists $\epsilon > 0$ such that \begin{align} \underset{u \in \mathbb{D}}{\sup}  \ B_\gamma (|f|^{2 + \epsilon}) (u) B_\gamma (|g|^{2 + \epsilon})(u)  < \infty.    \nonumber   \end{align} The authors then proved that $T_f T_{\overline{f^{-1}}}$ is bounded by showing that there exists some $\epsilon > 0$ where \begin{align} \underset{u \in \mathbb{D}}{\sup}  \ B_\gamma (|f|^{2 + \epsilon}) (u) B_\gamma (|f|^{-(2 + \epsilon)})(u)  < \infty  \nonumber    \end{align} whenever $(1.2)$ holds for $g = f^{-1}$ (which is true modulo a multiplicative constant if $(1.1)$ and $(1.2)$ hold.)  The boundedness of $T_f T_{\overline{g}}$ then follows easily from this fact and conditions $(1.1)$ and $(1.2)$. For the boundedness of the Toeplitz product $T_f T_{\overline{g}}$ on the Hardy space, the authors use the same argument and Theorem $8$ from \cite{Zh}.

\medskip
It was remarked in \cite{CMP}, however, that the boundedness of $T_f T_{\overline{g}}$ on either the Hardy space $H^2 (\partial \mathbb{D})$ or the Bergman space $L_a ^2 (\mathbb{D}, dv_\gamma)$  for analytic $f$ and $g$ is equivalent to the boundedness of the Hardy projection $P^+$ (respectively, the Bergman projection $P_\gamma$) from the weighted space $L^2  (\partial \mathbb{D}, |g|^{-2} d\sigma) $ to the weighted space $L^2(\partial \mathbb{D}, |f|^2 d\sigma)$ (where the obvious changes are made for the Bergman space.)

\noindent
More generally, the boundedness of the Hardy projection $P^+$ on $L^p (\partial \mathbb{B}_n, d\sigma)$ tells us that for any symbols $f$ and $g$ (not necessarily analytic),  $T_f T_{\overline{g}}$ is bounded on the Hardy space $H^p (\partial \mathbb{B}_n)$ (in fact, bounded on $L ^p (\partial \mathbb{B}_n, d\sigma))$  if $P^+$ is bounded from $L^p  (\partial \mathbb{B}_n, |g|^{-p} d\sigma) $ to $L^p(\partial \mathbb{B}_n, |f|^p d\sigma)$.  Moreover, a similar result holds for the boundedness of $T_f T_{\overline{g}}$ on $L_a ^p (\mathbb{B}_n, dv_\gamma)$.

\medskip
 Unfortunately, the ``two-weight" problem of characterizing the weights $w$ and $v$ on $\partial \mathbb{B}_n$  where $P^+$ is bounded from $L^p  (\partial \mathbb{B}_n, w \, d\sigma)$  to $L^p  (\partial \mathbb{B}_n, v \,  d\sigma) $ is very difficult and not fully understood even for $n = 1$ (the ``two-weight" problem for $P^+$ when $n = 1$ can be found in \cite{CS1, CS2}, but their condition is extremely difficult to work with and is thus far from optimal.)  Furthermore, a similar statement can be said about the corresponding problem for the Bergman projection on $\mathbb{B}_n$.

\medskip
 If $w = v$, however, it is well known that $P^+$ is bounded on $L^p  (\partial \mathbb{D}, w \, d\sigma)$ if and only if $w$ satisfies the Muckenhoupt A${}_p$ condition.  Similarly, it is well known that $P_\gamma$ is bounded on $L^p  (\mathbb{B}_n, w \, dv_\gamma)$ if and only if $w$ satisfies the B\`{e}koll\`{e} - Bonami condition $B_{p, \gamma}$ (both of these conditions will be defined in the next section.)

\medskip In the next section, we will combine weighted norm inequalities for the Hardy and Bergman projections with ideas from \cite{SZ1, SZ2} to characterize bounded and invertible $T_f T_{\bar{g}}$ on both the Hardy space $H^p (\partial \mathbb{D})$ and the Bergman space $L_a ^p (\mathbb{B}_n, dv_\gamma)$ when $f$ and $g$ are analytic. It should be noted that not only is this approach much simpler than the one taken in \cite{SZ1, SZ2}, but it also provides us with an approach that is potentially adaptable to other spaces.

\medskip
 In particular, in Section $3,$ we will characterize weights $w$ on $\mathbb{C}^n$ where the Fock projection (which will be defined in Section $3$) is bounded on the weighted space $\mathcal{L}_\alpha ^p (w)$.  Here,  $\mathcal{L}_\alpha ^p (w)$ is the Banach space (equipped with its canonical Banach space norm) of all $f$ where  $f(\cdot) e^{- \frac{\alpha}{2} |\cdot|^2} \in L^p(\mathbb{C}^n,  w \, dv)$ for $\alpha  > 0$. Also we will use the general arguments from Section $2,$ along with our weighted norm inequalities for the Fock projection, to characterize bounded and invertible Toeplitz products on $F_\alpha ^p$.

 As a trivial consequence of these results, we will show that ``Sarason's conjecture" on the product of Toeplitz operators is trivially true for the Fock space $F_\alpha ^p$, which is in stark contrast to the Hardy space where it is known that Sarason's conjecture is false (see \cite{CMP} for detailed information about Sarason's conjecture, and see \cite{N} for a counterexample in the Hardy space case).  In particular we prove that $T_f T_{\overline{g}}$ is bounded on $F_\alpha ^p$ if and only if $f = e^{q}$ and $g = c e^{-q}$ for some constant $c \in \mathbb{C}$ and some linear polynomial $q$.  Note that this was shown very recently in the preprint \cite{CPZ} using a simpler argument than the ones we employ here.  However, our arguments most likely work for a wide class of weighted Fock spaces, and in particular for the so called ``Fock-Sobolev spaces" from \cite{CCK} (see Section $3$ for more details).

\medskip
In Section $4$, we will discuss in some detail the various classes of weights used in Sections $2$ and $3$, and also discuss  connections between these classes.

\medskip
It should be noted that although the theory of weighted norm inequalities simplifies the arguments in \cite{SZ1, SZ2}, the techniques developed in these two papers (in particular, their ``reverse H\"{o}lder inequality" and the Calderon-Zygmund decomposition adapted to the hyperbolic disk) are of independent interest themselves.  Thus,  in our last section (Section $5$), we will present a proof of our characterization of invertible Toeplitz products on the Bergman space $L^p _a (\mathbb{D}, dA_\gamma)$ that extends these techniques to handle the general case $p > 1$, rather than just the $p = 2$ case.  In particular, we will extend the  ``reverse H\"older inequality" of Zheng and Stroethoff \cite{SZ1, SZ2} for $p = 2$ to the general case of $p > 1$.  It is hoped that the ideas in Section $5$ will have applications to other Bergman space problems where M\"{o}bius invariance is unavailable, or where classical Calderon-Zygmund theory techniques are relevant.

\medskip
Finally, throughout the paper we will let $C$ denote a constant that may change from line to line (or even on the same line.)

\section{ Invertible Toeplitz products on the Hardy and Bergman spaces.}

 We will first discuss invertible Toeplitz products on the Bergman space  $L_a ^p (\mathbb{B}_n, dv_\gamma)$.   The result we wish to prove is the following:

\begin{theorem} If $f \in L_a ^p (\mathbb{B}_n, dv_\gamma )$ and $g \in L_a ^q (\mathbb{B}_n, dv_\gamma)$, then the Toeplitz product $T_f T_{\overline{g}}$ is bounded and invertible on $L_a ^p (\mathbb{B}_n, dv_\gamma )$ if and only if \begin{align} \underset{u \in \mathbb{B}_n}{\inf}  \  |f (u)||g(u)| > 0 \nonumber  \end{align} and  \begin{align}  \underset{u \in \mathbb{B}_n}{\sup}  \ \{B_\gamma (|f k_u ^{ 1- \frac{2}{p}}|^p ) (u) \}^\frac{1}{p}  \{B_\gamma(|g k_u ^{ 1- \frac{2}{q}}|^q) (u) \}^\frac{1}{q}  < \infty.  \nonumber    \end{align} \end{theorem}

Before we prove this, we will need to discuss the B\`{e}koll\`{e} - Bonami class B${}_{p, \gamma}$.  For $z, u \in \mathbb{B}_n$, let $d$ be the pseudo-metric on $\mathbb{B}_n$ given by $d(z, u) = ||z| - |u|| + |1 - \frac{z}{|z|} \cdot \frac{u}{|u|}| $ and let $D = D(z, R)$ denote a ball in $\mathbb{B}_n$ with respect to this pseudo-metric. We say that a weight $w$  on $\mathbb{B}_n$ is in B${}_{p, \gamma}$ if \begin{align}  \left(\frac{1}{v_\gamma (D)} \int_D w \, dv_\gamma \right) \left(\frac{1}{v_\gamma (D)} \int_D w^{ - \frac{1}{p - 1}} \, dv_\gamma \right)^{p - 1} < C \tag{2.1} \end{align} where $D$ is any such ball that intersects $\partial \mathbb{B}_n$ and $C$ is independent of $D$.

\medskip
 The following theorem was proved in \cite{Be}, which solves the ``one-weight" problem for the Bergman projection $P_\gamma$: \begin{theorem} The Bergman projection $P_\gamma$ is bounded on the weighted space $L^p (\mathbb{B}_n, w \, dv_\gamma)$ if and only if $w \in \text{B}_{p, \gamma}$. \end{theorem}

\noindent We will also need the following result found in \cite{M} \begin{theorem} If $f \in L_a ^p (\mathbb{B}_n, dv_\gamma)$ and $g \in L_a ^q (\mathbb{B}_n, dv_\gamma),$ then  \begin{align}  \underset{u \in \mathbb{B}_n}{\sup}  \ \{B_\gamma (|f k_u ^{ 1- \frac{2}{p}}|^p ) (u) \}^\frac{1}{p}  \{B_\gamma (|g k_u ^{ 1- \frac{2}{q}}|^q) (u) \}^\frac{1}{q}  < \infty  \nonumber  \end{align} whenever $T_f T_{\overline{g}}$ is bounded on $L_a ^p (\mathbb{B}_n, dv_\gamma)$. \end{theorem}

\noindent With the aid of Theorem $2.2$ and $2.3$,  we can now prove Theorem $2.1$:

\medskip
\noindent  \textit{Proof of Theorem} $2.1$:  \  First we will prove necessity. The proof of this direction is similar to the corresponding result in \cite{SZ1, SZ2}, though we include it for the sake of completion.    Assume that $T_f T_{\overline{g}}$ is bounded and invertible on $L_a ^p (\mathbb{B}_n, dv_\gamma)$, so that $(T_f T_{\overline{g}})^* = T_g T_{\overline{f}}$ is bounded and invertible on $L_a ^q (\mathbb{B}_n, dv_\gamma)  = \left(L_a ^p (\mathbb{B}_n, dv_\gamma) \right)^*$.  Let  $C_1 = \|(T_f T_{\overline{g}})^{-1}\|_p$ and $ C_2 =  \|(T_g T_{\overline{f}})^{-1}\|_q. $   First note that $T_f T_{\overline{g}} k_u = \overline{g(u)} f k_u$, so that \begin{align} \|k_u\|_p  & \leq C_1 \|T_f T_{\overline{g}} k_u\|_p  \nonumber   \\ & =  C_1 |g(u)| \{B_\gamma (|f k_u ^{1 - 2/p}|^p ) (u) \} ^\frac{1}{p}  \nonumber  \end{align}
Similarly, since $\left(T_f T_{\overline{g}}\right)^* = T_g T_{\overline{f}}$ is bounded and invertible on $L_a ^q (\mathbb{B}_n, dv_\gamma)  = \left(L_a ^p (\mathbb{B}_n, dv_\gamma) \right)^*$, we have that \begin{align} \|k_u\|_q \leq C_2   |f(u)| \{B_\gamma (|g k_u ^{1 - 2/q}|^q )  (u)\}^\frac{1}{q}.  \nonumber \end{align} By Theorem $2.2$, we have that \begin{align} \{B_\gamma (|f k_u ^{1 - 2/p}|^p ) (u) \} ^\frac{1}{p}   \{B_\gamma (|g k_u ^{1 - 2/q}|^q ) (u)\} ^\frac{1}{q}   \leq M \tag{2.2}\end{align} for some $M > 0$ independent of $u$.  Moreover, an application of H\"{o}lder's inequality gives us that $\|k_u\|_p \|k_u\|_q \geq 1$ for any $u \in \mathbb{D}$, which tells us that \begin{align}  C_1 C_2 M |f(u)| |g(u)|  \geq \|k_u\|_p \|k_u\|_q  \geq 1 \nonumber \end{align} which means that  $\underset{u  \in \mathbb{B}_n}{\inf} \ |f(u)||g(u)| > 0. $

\bigskip Now we will prove sufficiency. Let $M$ be the constant in $(2.2)$ and let \begin{align} \eta = \underset{u \in \mathbb{B}_n}{\inf}  \ |f(u)| |g(u)|. \nonumber \end{align}

\noindent Let $\varphi_u$ be the M\"{o}bius transformation that interchanges $0$ and $u$.  By H\"{o}lder's inequality, we have that \begin{align} |f(u)| & = (1 - |u|^2)^{\left(\frac{n + 1 + \gamma}{2}\right)\left(1 - \frac{2}{p}\right)} |f \circ \varphi_u (0)| |k_u \circ \varphi_u (0)| ^{1 - \frac{2}{p}} \nonumber \\  & \leq (1 - |u|^2)^{\left(\frac{n + 1 + \gamma}{2}\right)\left(1 - \frac{2}{p}\right)} \{B_\gamma( |f  k_u ^{1 - 2/p}|^p ) (u) \}^\frac{1}{p}. \tag{2.3} \end{align}  and similarly \begin{align} |g(u)| \leq (1 - |u|^2)^{\left(\frac{n + 1 + \gamma}{2}\right)\left(1 - \frac{2}{q}\right)} \{B_\gamma( |g  k_u ^{1 - 2/q}|^q ) (u) \}^\frac{1}{q} \nonumber \end{align} which means that \begin{align} \underset{u \in \mathbb{B}_n}{\sup} \ |f(u)| |g(u)| \leq M \nonumber \end{align}

\noindent
Also, since $|g(u)|^q \geq \eta^q |f^{-1} (u)|^q $ we have that \begin{align} \{B_\gamma ( |f^{-1} k_u ^{1 - 2/q}| ^q )(u) \}^\frac{1}{q} \leq \eta^{-1} \{B_\gamma ( |g k_u ^{1 - 2/q}| ^q )(u) \}^\frac{1}{q} \nonumber \end{align} which means that \begin{align}  \underset{u \in \mathbb{B}_n}{\sup}  \ \{B_\gamma (|f k_u ^{ 1- \frac{2}{p}}|^p ) (u) \}^\frac{1}{p}  \{B_\gamma(|f^{-1} k_u ^{ 1- \frac{2}{q}}|^q) (u) \}^\frac{1}{q}  < \infty. \tag{2.4}   \end{align}

 If $w = |f|^p$, then it is easy to see that $(2.4)$ and Lemma $2$ in \cite{Be} implies that $w \in \text{B}_{p, \gamma}$, so that $T_f T_{\overline{g}}$ is bounded on $L_a ^p (\mathbb{B}_n, dv_\gamma)$.  Also, since $\phi = (f \overline{g})^{-1}$ is bounded, $T_\phi$ is bounded on $L_a ^p (\mathbb{B}_n, dv_\gamma)$. Moreover, it is easy to check that \begin{align} T_f T_{\overline{g}} T_\phi = I = T_\phi T_f T_{\overline{g}} \nonumber \end{align} which completes the proof. \hfill $\square$

\medskip
We will now prove the Hardy space version of Theorem $2.1$.  First, recall that the Muckenhoupt class A${}_p$ is the collection of all weights $w$ on $\partial \mathbb{D}$ where \begin{align} \underset{I \subseteq \partial \mathbb{D}}{\sup} \ \left(\frac{1}{|I|} \int_I w \, d\theta \right) \left(\frac{1}{|I|} \int_I w^{ - \frac{1}{p - 1}} \, d\theta \right)^{p - 1} < \infty \tag{2.5}  \end{align} and where  the supremum is taken over all arcs $I \subseteq \partial \mathbb{D}$.  It is well known \cite{CF} that the Hardy projection $P^+$ is bounded on $L^p(\partial \mathbb{D}, w \, d\theta)$ if and only if $w \in \text{A}_p$.  With this result, we will now prove the following:

 \begin{theorem}  If $f \in H ^p (\partial{\mathbb{D}})$ and $g \in H ^q (\partial{\mathbb{D}})$, then the Toeplitz product $T_f T_{\overline{g}}$ is bounded and invertible on $H ^p (\partial{\mathbb{D}})$ if and only if \begin{align} \underset{u \in \mathbb{D}}{\inf}  \  |f (u)||g(u)| > 0  \tag{2.6} \end{align} and  \begin{align}  \underset{u \in \mathbb{D}}{\sup}  \ \{\widehat{|fk_u ^{1 - \frac{2}{p}} |^p } (u) \} ^\frac{1}{p} \{\widehat{|g k_u ^{1 - \frac{2}{q}}|^q} (u) \}^\frac{1}{q}  < \infty.  \nonumber \end{align} \end{theorem}

\begin{proof} First we prove necessity, so assume that $T_f T_{\overline{g}}$ is bounded and invertible on $H^p(\partial \mathbb{D})$.  By an argument that is almost identical to the argument (due to S. Treil) in \cite{S}, we have that \begin{align}  \underset{u \in \mathbb{D}}{\sup}  \ \{\widehat{|fk_u ^{1 - \frac{2}{p}} |^p } (u) \} ^\frac{1}{p} \{\widehat{|g k_u ^{1 - \frac{2}{q}}|^q} (u) \}^\frac{1}{q}  < \infty.   \tag{2.7} \end{align} Thus, by an argument that is similar to the proof of Theorem $2.1$, we have that  \begin{align} \underset{u  \in \mathbb{D}}{\inf} \ |f(u)||g(u)| > 0. \nonumber \end{align}

\medskip
Now we prove sufficiency.  Fix $u \in \mathbb{D}$ and let $f_r (u) = f(ru)$ for some fixed $0 < r < 1$.  If we replace $f$ with $f_r$ in $(2.3)$ then since $dA_\gamma \stackrel{\text{wk}^*}{\longrightarrow} d\sigma$ on $C(\overline{\mathbb{D}})$ as $\gamma \rightarrow -1^{+}$, $(2.3)$ gives us that \begin{align} |f(ru)| \leq (1 - |u|^2)^{\frac{1}{2}\left(1 - \frac{2}{q}\right)} \{\widehat{ |f_r  k_u ^{1 - \frac{2}{p}}|^p } (u) \}^\frac{1}{p}. \nonumber \end{align} Thus, since $f  \in H^p (\partial \mathbb{D})$, we can let $r \rightarrow 1^-$ to get \begin{align} |f(u)| \leq (1 - |u|^2)^{\frac{1}{2}\left(1 - \frac{2}{q}\right)} \{\widehat{ |f  k_u ^{1 - \frac{2}{p}}|^p } (u) \}^\frac{1}{p}. \nonumber \end{align} Applying the same inequality to $g$ and using the hypothesis of Theorem $2.4$, we  have that \begin{align} \underset{u \in \mathbb{D}}{\sup} \ |f(u)||g(u)| < \infty. \tag{2.8} \end{align}

\noindent Combining $(2.6), (2.7)$ and $(2.8)$ as we did in the proof of Theorem $2.1$, it is easy to see that $w = |f|^p$ is in the Muckenhoupt A${}_p$ class, which implies that $T_f T_{\overline{g}}$ is bounded.  Finally, since $\phi = (f\overline{g})^{-1} $ is bounded, it is again easy to see that \begin{align} T_f T_{\overline{g}} T_\phi = I = T_\phi T_f T_{\overline{g}} \nonumber \end{align} which implies that $T_f T_{\overline{g}}$ is invertible.
\end{proof}

\noindent
\begin{remark} It is known \cite{LR} that the Hardy projection $P^+$ is bounded on $L^p(\partial \mathbb{B}_n, w \, d\sigma)$ if and only if $w$ satisfies $(2.5)$ (where the supremum is taken over all non-isotropic balls in $\partial \mathbb{B}_n$).  Furthermore, except for the proof that the boundedness of $T_f T_{\overline{g}}$ implies $(2.7)$ (which uses identities that only hold when $n = 1$, see \cite{S} for more details), the entire proof of Theorem $2.4$ carries over to the case $n > 1$.  Thus, we will conjecture that Theorem $2.4$ holds for the unit sphere $\partial \mathbb{B}_n$ when $n > 1$. \end{remark}

\section{Invertible Toeplitz products and weighted norm inequalities for the Fock projection.}

For any $\alpha > 0$ and any positive a.e. function $w$ on $\mathbb{C}^n$, let $\mathcal{L}_\alpha ^p (w)$ be the Banach space of all $f$ where $f(\cdot) e^{- \frac{\alpha}{2} |\cdot|^2} \in L^p(\mathbb{C}^n, w dv)$.  Furthermore, we will let $\mathcal{L}_\alpha ^p$ denote $\mathcal{L}_\alpha ^p (w)$ when $w$ is the constant $\left(p \alpha/ 2 \pi\right)^n$.    It is well known (see \cite{JPR})  that the orthogonal projection $P_\alpha $ from $\mathcal{L}_\alpha ^2$ onto the Fock space $F_\alpha ^2 $ is given by \begin{align} P_\alpha f (z) = \int_{\mathbb{C}^n} e^{\alpha z \cdot u} f(u) d\mu_\alpha(u) \nonumber \end{align} where  $d\mu_\alpha$ is the Gaussian measure \begin{align} d\mu_\alpha (u) = \left(\frac{\alpha}{\pi}\right)^n e^{-\alpha |u|^2} dv(u). \nonumber \end{align}

\medskip
In this section, we will first state and prove weighted norm inequalities for the Fock projection $P_\alpha$, and then use these weighted norm inequalities to characterize bounded and invertible Toeplitz products $T_f T_{\overline{g}}$ on the Fock space $F_\alpha ^p$ for $p > 1$ when $f$ and $g$ are entire. In particular, as was stated in the introduction, we will show that $T_f T_{\overline{g}}$ is bounded if and only if $f = e^P$ for a linear polynomial $P$ on $\mathbb{C}^n$ and $g = c e^{-P}$ for some $c \in \mathbb{C}$ (assuming neither $f$ nor $g$ vanish on $\mathbb{C}^n$).  Furthermore, as was mentioned in the introduction, although this result (with a simpler proof) has recently appeared in \cite{CPZ}, our methods most likely easily generalize to a wide class of weighted Fock spaces.

\medskip
  Let $Q_r (z)$ be the cube in $\mathbb{C}^n$ with center $z$ and side length $r > 0$.  Let A${}_{p, r}  $ denote the class of weights $w$ on $\mathbb{C}^n$ where  \begin{align} \underset{z \in \mathbb{C}^n}{\sup} \left(\frac{1}{v(Q_r(z))} \int_{Q_r(z)} w \, dv \right)\left(\frac{1}{v(Q_r(z) )} \int_{Q_r(z) } w ^{- \frac{1}{p - 1}} \, dv \right) ^{p - 1} < C_r \tag{3.1} \end{align} for some $0 < C_r < \infty$.

  \begin{theorem} The following are equivalent for any weight $w$ on $\mathbb{C}^n$ and any $\alpha > 0$:
\begin{enumerate}
\item[(a)] $w \in \text{A}_{p, r}$ for some $r > 0$.
\item[(b)] $H_\alpha$ is bounded on $L^p(\mathbb{C}^n, w \, dv)$.
\item[(c)] $P_\alpha$ is bounded on $\mathcal{L}_\alpha^p (w)$
\item[(d)] $w \in \text{A}_{p, r}$  for all $r > 0$.
\end{enumerate}

\noindent Here, $H_\alpha$ is the integral operator given by \begin{align} H_\alpha f (z) =  \int_{\mathbb{C}^n} e^{- \frac{\alpha}{2} |z - u|^2} f(u) \, dv(u). \nonumber \end{align}

\end{theorem}

 We will need three simple lemmas to prove Theorem $3.1$. It should be noted that the proofs of the first two lemmas use standard arguments from the classical theory of weighted norm inequalities. In what follows, we will let \begin{align} w(S) := \int_{S} w \, dv \nonumber \end{align} for any measurable $S \subseteq \mathbb{C}^n$.

\begin{lemma} Let $Q_r = Q_r(z)$ be any cube in $\mathbb{C}^n$ of side length $r$, and let $3Q_r$ denote the cube with the same center but with side length $3r$.  If $w \in \text{A}_{p, 3r}$, then $w(3Q_r) \leq C w(Q_r)$ for some constant $C > 0$ independent of $Q_r$ (but obviously depending on $r$.) \end{lemma}

\begin{proof} By H\"{o}lder's inequality and $(3.1)$, there exists $C > 0$ such that \begin{align} r^{2n} & = \int_{Q_r} w ^ {1/p} w^{-1/p} \, dv \nonumber \\ & \leq (w(Q_r))^\frac{1}{p} \left(\int_{Q_r} w ^{- \frac{1}{p - 1}} \, dv \right)^{(p - 1)/p} \nonumber \\ & \leq  \left(\frac{w(Q_r)}{w(3Q_r)}\right)^\frac{1}{p} \left(\int_{3Q_r} w \, dv \right)^\frac{1}{p} \left(\int_{3Q_r} w ^{- \frac{1}{p - 1}} \, dv \right)^{(p - 1)/p} \nonumber \\ & \leq C \left(\frac{w(Q_r)}{w(3Q_r)}\right)^\frac{1}{p} \nonumber \end{align} where $C $ is independent of $Q_r$.   \end{proof}

\begin{lemma} Let $Q_r$ be any cube in $\mathbb{C}^n$ of side length $r$ and let $f$ be any measurable function on $\mathbb{C}^n$.  If $w \in \text{A}_{p, 3r}$, then there exists $C > 0$ independent of $Q_r$ and $f$ where \begin{align} \left(\int_{Q_r} |f| \, dv \right)^p \leq C \frac{1}{w(Q_r) } \int_{Q_r} |f|^p w \, dv \nonumber \end{align} \end{lemma}

\begin{proof} The proof is similar to the proof of Lemma $3.2$.  In particular, since A${}_{p, 3r} \subseteq \text{A}_{p, r}$, there is some $C > 0$ independent of $Q_r$ where \begin{align} \left(\int_{Q_r} |f| \, dv \right)^p & \leq \left(\int_{Q_r} |f|^p w \, dv \right)\left(\int_{Q_r} w ^{- \frac{1}{p - 1}} \, dv \right)^{p - 1} \nonumber \\ & \leq  C \frac{1}{w(Q_r) } \int_{Q_r} |f|^p w \, dv \nonumber \end{align} \end{proof}

For the next lemma we will need the notion of a discrete path from \cite{I}.  For each $r > 0$, let $r \mathbb{Z}^{2n} $ denote the set $\{(r k_1, \ldots, r k_{2n}) \in \mathbb{R}^{2n} : k_i \in \mathbb{Z}  \}$. Since $\mathbb{R}^{2n}$ can canonically be identified with $\mathbb{C}^n$, we will treat $r \mathbb{Z}^{2n}$ as a subset of $\mathbb{C}^n$.   A subset $G = \{p_0, \ldots, p_k\}$ of $r \mathbb{Z}^{2n}$ with $k \geq 1$ is said to be a discrete segment in $r \mathbb{Z}^{2n}$ if there exists $j \in \{1, \ldots, 2n\}$ and $z
\in r\mathbb{Z}^{2n}$ such that \begin{eqnarray} p_\ell = z + \ell   (r e_j), \hspace{1cm} 0 \leq \ell \leq k \nonumber \end{eqnarray}
where $e_j$ is the standard $j ^{\text{ th }}$ basis vector of $\mathbb{R}^{2n}$.  In this setting, we say that $p_0$ and $p_k$ are the endpoints of $G$.  Also, we define the length $|G|$ of $G$ to be $|G| = k$.  Let $\nu = (r \nu_1, \ldots, r \nu_{2n})$ and $\nu' = (r \nu_1 ', \ldots, r \nu_{2n} ')$ be elements of $r \mathbb{Z}^{2n}$ where $\nu \neq \nu'$. We can enumerate the integers $\{j : \nu_j \neq {\nu}_j ', 1\leq j \leq 2n\}$ as $j_1, \ldots, j_m$ in ascending order, so that $j_1 < \cdots < j_m$ when $m > 1$.   Set $z_0 (\nu, \nu') = \nu$, and inductively define $z_t(\nu, \nu') = z_{t - 1} (\nu, \nu') + (\nu_{j_t} ' - \nu_{j_t}) (r e_{j_t})$ for $ t \in \{1, \ldots, m\}.$  Note that $z_m(\nu, \nu') = \nu'$.  Let $G_t(\nu, \nu')$ be the discrete segment in $r \mathbb{Z}^{2n} $ which has $z_{t - 1}(\nu, \nu')$ and $z_t (\nu, \nu')$  as its endpoints. The union of the discrete segments $G_1(\nu, \nu'), \ldots, G_m(\nu, \nu')$ will be denoted by $\Gamma(\nu, \nu')$.  We call $\Gamma(\nu, \nu')$ the discrete path in  $r \mathbb{Z}^{2n}$ from  $\nu$ to  $\nu'$.  Furthermore, we define the length $|\Gamma(\nu, \nu')|$ of $\Gamma(\nu, \nu')$ to be $|G_1(\nu, \nu')| + \cdots + |G_m(\nu, \nu')|.$ That is, the length of $\Gamma(\nu, \nu')$ is just the sum of the lengths of the discrete segments which make up $\Gamma(\nu, \nu')$.
In the case $\nu = \nu'$, we define the discrete path from $\nu$ to $\nu$ to be the singleton set $\Gamma(\nu, \nu) = \{\nu\}$.

\begin{lemma}  If $w \in \text{A}_{p, 3r}$ then there exists $C > 0$ independent of $\nu, \nu' \in r \mathbb{Z}^{2n}$ such that \begin{align} \frac{w( Q_r( \nu))}{w( Q_r (\nu'))} \leq C ^{|\nu - \nu'|} \nonumber \end{align} \end{lemma}

\begin{proof}  Enumerate the elements in $\Gamma(\nu, \nu')$ as $a_0, a_1, \ldots, a_k$ where $a_0 = \nu, \ a_k = \nu', \ k = |\Gamma(\nu, \nu')|,$ and \begin{align} Q_r ( a_{j-1}) \subseteq 3Q_r(  a_{j }) \nonumber \end{align} for each $j \in \{1, \ldots, k\}$. Then by Lemma $3.2$, there exists $C > 0$ where \begin{align} \frac{w(Q_r ( \nu))}{w(Q_r (\nu'))} & = \prod_{j = 1}^{k} \frac{w(Q_r (a_{j - 1}))}{w(Q_r ( a_j))} \nonumber \\ & \leq \prod_{j = 1}^{k} \frac{w(3Q_r ( a_{j }))}{w(Q_r ( a_j))} \nonumber \\ & \leq C^{|\Gamma(\nu, \nu')|}. \nonumber \end{align}  However, an easy application of the Cauchy-Schwarz inequality tells us that $|\Gamma(\nu, \nu')| \leq \frac{(2n)^\frac{1}{2} |\nu - \nu'|}{r}, $ which completes the proof.  \end{proof}

\medskip
\noindent We will now prove Theorem $3.1$.

\smallskip
\noindent \textit{Proof of Theorem $3.1$: }We will first prove that $(a) \Rightarrow (b) \Rightarrow (c) \Rightarrow (a)$. Then since trivially $(d) \Rightarrow (a)$, we will complete the proof by showing that $(b) \Rightarrow (d)$.

\medskip
Let $r' = \frac{1}{3} r$.  To show that $(a) \Rightarrow (b)$, we have:    \begin{align} \|H_\alpha & f  \|_{L^p(\mathbb{C}^n, w \, dv)} ^p  \nonumber \\ & \leq \int_{\mathbb{C}^n} \left(\int_{\mathbb{C}^n } e^{- \frac{\alpha}{2} |z - u|^2 } |f(u)| \, dv(u) \right)^p w(z) \, dv(z) \nonumber \\ & = \sum_{\nu \in r' \mathbb{Z}^{2n}} \int_{Q_{r'}(\nu)} \left(\sum_{\nu' \in r'\mathbb{Z}^{2n} } \int_{ Q_{r'} (\nu')} e^{- \frac{\alpha}{2} |z - u|^2} |f(u)| \, dv(u) \right)^p w(z) \, dv(z) \nonumber \\ & \leq C \sum_{\nu \in r' \mathbb{Z}^{2n}} \int_{Q_{r'}(\nu)} \left(\sum_{\nu' \in r'\mathbb{Z}^{2n} } e^{- \frac{\alpha}{4} |\nu - \nu'|^2}\int_{ Q_{r'} (\nu')}  |f(u)| \, dv(u) \right)^p w(z) \, dv(z) \nonumber \\ & = \sum_{\nu \in r' \mathbb{Z}^{2n}}  w(Q_{r'}(\nu))  \left(\sum_{\nu' \in r'\mathbb{Z}^{2n} } e^{- \frac{\alpha}{4} |\nu - \nu'|^2}  \int_{ Q_{r'} (\nu')}  |f(u)| \, dv(u) \right)^p  \nonumber \end{align}

\medskip
\noindent By H\"{o}lder's inequality, we have \begin{align} \sum_{\nu \in r' \mathbb{Z}^{2n}} &  w(Q_{r'}(\nu)) \left(\sum_{\nu' \in r'\mathbb{Z}^{2n} } e^{- \frac{\alpha}{4} |\nu - \nu'|^2}  \int_{ Q_{r'} (\nu')}  |f(u)| \, dv(u) \right)^p \nonumber \\ & \leq C \sum_{\nu \in r' \mathbb{Z}^{2n}} w(Q_{r'}(\nu))   \sum_{\nu' \in r'\mathbb{Z}^{2n} } e^{- \frac{p\alpha}{8} |\nu - \nu'|^2}  \left(\int_{ Q_{r'} (\nu')}  |f(u)| \, dv(u) \right)^p  \tag{3.2} \end{align}

\noindent
However, since $w \in \text{A}_{p, 3r'}$, Lemmas $3.3$ and $3.4$ give us that \begin{align} w(Q_{r'}(\nu))  \left(\int_{ Q_{r'} (\nu')}  |f| \, dv \right)^p & \leq C \frac{w(Q_{r'} (\nu))} {w(Q_{r'} (\nu'))} \int_{Q_{r'}(\nu')} |f|^p w \, dv \nonumber  \\ & \leq C ^{|\nu - \nu'| + 1}  \int_{Q_{r'}(\nu')} |f|^p w \, dv \tag{3.3}  \end{align}

\noindent Plugging $(3.3)$ into $(3.2)$ and switching the order of summation, we have that \begin{align} \sum_{\nu \in r' \mathbb{Z}^{2n}} w(Q_{r'}(\nu))  & \sum_{\nu' \in r'\mathbb{Z}^{2n} } e^{- \frac{p\alpha}{8} |\nu - \nu'|^2}  \left(\int_{ Q_{r'} (\nu')}  |f(u)| \, dv(u) \right)^p  \nonumber \\ & \leq \sum_{\nu \in r' \mathbb{Z}^{2n}}    \sum_{\nu' \in r'\mathbb{Z}^{2n} }  C ^{|\nu - \nu'| + 1} e^{- \frac{p\alpha }{8} |\nu - \nu'|^2}  \int_{Q_{r'}(\nu ')} |f|^p w \, dv \nonumber \\ & = \sum_{\nu' \in r' \mathbb{Z}^{2n}}    \int_{Q_{r'}(\nu')} |f|^p w \, dv \sum_{\nu \in r'\mathbb{Z}^{2n} }  C ^{|\nu - \nu'| + 1} e^{- \frac{p\alpha }{8} |\nu - \nu'|^2} \nonumber \\ & \leq C \int_{\mathbb{C}^n} |f|^p w \, dv \nonumber \end{align}
\medskip

That $(b) \Rightarrow (c)$ follows from a simple computation.

\medskip
  Let us now prove that $(c) \Rightarrow (a)$. The proof will involve a modification of the proof of the corresponding result in \cite{CF} for the Hilbert transform on the weighted space $L^p(\mathbb{R}, w \, dx)$. Fix some cube $Q$ with center $z_0$ and side length $ r_0$ where $r_0 > 0$ is a small number to be determined. If \begin{align} f(u) = w^{- \frac{1}{p - 1}} (u) e^{\frac{\alpha}{2} |u|^2} e^{- i \alpha \text{Im} ( z_0 \cdot u) } \chi_{Q} (u), \nonumber \end{align}   \noindent then \begin{align} |P_\alpha f(z)| = \left(\frac{\alpha}{\pi}\right)^n \left|\int_{Q } e^{\alpha (z \cdot u)} e^{- \frac{\alpha}{2} |u|^2 } e^{- i \alpha \text{Im} ( z_0 \cdot u) } w^{- \frac{1}{p - 1}} (u) \, dv(u) \right| \tag{3.4} \end{align}

\noindent  However, \begin{align} e^{\alpha (z \cdot u)} & = \left|e^{\alpha (z \cdot u)}\right|e^{i \alpha \text{Im}(z\cdot u)} \nonumber \\ & =  \left|e^{\alpha (z \cdot u)}\right|e^{i \alpha \text{Im} (z - z_0) \cdot (u - z_0)} e^{i \alpha \text{Im} (z_0 \cdot u)} e^{i \alpha \text{Im} (z - z_0)\cdot z_0} \tag{3.5} \end{align}

\noindent Plugging $(3.5)$ into $(3.4)$ gives \begin{align} |P_\alpha f (z)| = \left(\frac{\alpha}{\pi}\right)^n  e^{\frac{\alpha}{2} |z|^2} \left|\int_{Q } e^{-\frac{\alpha}{2} |z - u|^2} e^{i \alpha \text{Im} (z - z_0) \cdot (u - z_0)} w^{- \frac{1}{p - 1}} (u) \, dv(u) \right| \nonumber \end{align}

\noindent  Picking $r_0 > 0$ small enough, we get that $\left|  1 - e^{i \alpha \text{Im} (z- z_0)\cdot (u - z_0)} \right| \leq \frac{1}{2}  $ for all $z$ and $ u \in Q$, so writing $ e^{i \alpha \text{Im} (z- z_0)\cdot (u - z_0)} = 1 - \left(   1 - e^{i \alpha \text{Im} (z- z_0)\cdot (u - z_0)}\right)$ and using the triangle inequality, we get that \begin{align} |P_\alpha f (z)| & \geq \frac{1}{2} \left(\frac{\alpha}{\pi}\right)^n e^{ \frac{\alpha}{2} |z|^2} \chi_Q (z)  \int_{Q }  e^{- \frac{\alpha}{2} |z - u|^2 } w^{- \frac{1}{p - 1}} (u)  \, dv(u) \nonumber \\ &  \geq C e^{ \frac{\alpha}{2} |z|^2} \chi_Q (z) \int_{Q }  w^{- \frac{1}{p - 1}} \, dv. \tag{3.6}  \end{align}   The boundedness of $P_\alpha$ on $\mathcal{L}_\alpha ^p (w)$ applied to $(3.6)$ now gives us that \begin{align}
w(Q) \left( \int_Q w^{- \frac{1}{p - 1}} \, dv\right)^p \leq C \int_Q w^{- \frac{1}{p - 1}}  \, dv \nonumber \end{align} which proves $(a)$. \noindent Finally, the proof that $(b) \Rightarrow (d)$ is similar to the proof that $(c) \Rightarrow (a)$. \hfill $\square$

\bigskip

\bigskip

 \begin{remark} By Theorem $3.1$ we have that the classes A${}_{p, r}$ coincide for each $r > 0$.  Thus, to emphasize this fact, we will denote the space A${}_{p, r}$ by A${}_p ^\text{restricted}$.\\

\noindent Also, since A${}_p ^\text{restricted}$ is obviously independent of $\alpha$, we have that $P_{\alpha_0}$ is bounded on $\mathcal{L}_{\alpha_0} ^p (w)$ for \textit{some} $\alpha_0 > 0$ if and only if $P_{\alpha}$ is bounded on $\mathcal{L}_{\alpha} ^p (w)$ for \textit{all} $\alpha > 0$. \end{remark}

\begin{remark} The definition of A${}_p ^\text{restricted}$ can obviously be defined on $\mathbb{R}^n$ for all $n \in \mathbb{N}$.  Moreover, we also have that A${}_p ^\text{restricted}$ is the same as the class of weights $w$ on $\mathbb{R}^n$ where $H_\alpha$ is bounded on $L^p(\mathbb{R}^n, w \, dv)$ for any (or all) $\alpha > 0$ . \end{remark}

We will now connect the class A${}_p ^\text{restricted}$ with an appropriate BMO type space.  For $1 \leq p < \infty$,  let $\text{BMO}_r ^p $ be the space of functions $f$ on $\mathbb{R}^n$ such that

\begin{eqnarray}
\underset{z \in  {\mathbb{R}^n}}{\sup}  \frac{1}{v(B(z, r))} \int_{B(z, r)} \left| f  - f_{B(z, r)}  \right|^p dv < \infty \nonumber \end{eqnarray} where $B(z, r)$ is a Euclidean ball of center $z \in \mathbb{R}^n$ and radius $r > 0$.  It is easy to show that as a vector space, $\text{BMO}_r ^p $  is independent of $r > 0$, and so we will write $\text{BMO} ^p $ instead of $\text{BMO}_r ^p $.  It is also not hard to show that $\text{BMO} ^p = \text{BA}^p + \text{BO}$ where  $f \in \text{BO}$ if

\begin{eqnarray} \underset{z \in \mathbb{R}^n}{\sup} \omega_r (f)(z) < \infty  \nonumber
\end{eqnarray}

\noindent for some (or any) fixed $r > 0$ where $\omega_r (f)(z) = \underset{ w \in B(z, r)}{\sup} |f(z) - f(w)|$ and $f \in \text{BA} ^p $ if

\begin{eqnarray} \underset{z \in \mathbb{R}^n}{\sup} \frac{1}{v(B(z, r))} \int_{B(z, r)} |f|^p dv  < \infty. \nonumber
\end{eqnarray}

\noindent for some (or any) fixed $r > 0$.  Note that both of these conditions are independent of $r > 0$. Also note that this decomposition is explicit.  In particular, if $f \in \text{BMO} ^p$, then one can verify that $f_{B(\cdot, r)} \in \text{BO}$ and $f - f_{B(\cdot, r)} \in \text{BA} ^p$ for any $r > 0.$     Unlike in the classical BMO setting, note that the John-Nirenberg theorem is not true for the spaces $\text{BMO} ^p$ since the space $\text{BA}^p$ depends on $p$.  For more details about $\text{BMO} ^p$ (and for proofs of the above assertions) see \cite{CIL}, p. 3023.

However, similar to the classical BMO setting, one can show that $\log w \in \text{BMO}^1$ if $w \in \text{A}_p ^\text{restricted}$, where the proof is identical to the proof in the classical A${}_p$ - BMO setting (see \cite{D}, p. 151.) It is also well known that in the classical setting, $e^{\delta f} \in \text{A}_p$ for $f \in \text{BMO}$ with $\delta > 0$ small enough (again see \cite{D} p. 151).  It would be interesting to know if any similar relationship between $\text{BMO}^p$ and $\text{A}_p ^\text{restricted}$ exists.

     \medskip
      With Theorem $3.1$ proved, we can now characterize invertible Toeplitz products on the Fock space. In fact, we will characterize bounded Toeplitz products $T_f T_{\overline{g}}$ when $f, g$ are entire and as a consequence, as mentioned before, we will show that Sarason's conjecture is trivially true for the Fock space. First, for a function $f$ on $\mathbb{C}^n$,  let $\tilde{f}^{(\alpha)}$ be the Berezin transform of $f$ given by \begin{align} \tilde{f}^{(\alpha)} (z) =  \left(\frac{\alpha}{\pi}\right)^n \int_{\mathbb{C}^n} e^{- \alpha |z - u|^2} f(u) \, dv(u). \nonumber \end{align}   Note that $\tilde{f}^{(\alpha)}$ can obviously be defined for a function $f$ on $\mathbb{R}^n$. Moreover, for a function $f$ on $\mathbb{R}^n$, note that $\tilde{f}^{(\alpha)}$ is just the convolution of $f$ with the heat kernel $H(x, t) = \frac{1}{(4\pi t )^{\frac{n}{2}}} \exp \{-|x|^2/{4t}\}$ at time $t = \frac{1}{4\alpha}$.

       \begin{theorem} Let $p > 1$ and let $q$ be the conjugate exponent of $p$.  If $f \in F_\alpha ^p $ and $g \in F_\alpha ^q $, then the following are equivalent:
\begin{enumerate}
\item[(a)] $T_f  T_{\overline{g}}  $ is bounded on $F _{\alpha} ^p $.
\item[(b)] $f$ and $g$ satisfy \begin{align} \underset{z \in \mathbb{C}^n}{\sup}  \left(\widetilde{|f|^p}^{\left(\frac{\alpha p}{2}\right)} (z) \right)^\frac{1}{p} \left(\widetilde{|g|^{q}}^{\left(\frac{\alpha q}{2}\right)} (z) \right)^\frac{1}{q} < \infty.  \nonumber \end{align}
\end{enumerate}

\noindent Furthermore, if either of these are true then $fg$ is identically constant, and if both $f$ and $g$ never vanish on $\mathbb{C}^n$, then $f = e^{P}$ for a linear polynomial $P$ on $\mathbb{C}^n$ and $g = c e^{-P}$ for some $c \in \mathbb{C}$ (in which case $T_f  T_{\overline{f^{-1}}}$  is invertible on $F_\alpha ^p$). \end{theorem}

\begin{proof}   We first prove that $(a) \Longrightarrow (b)$.  Assume that $T_f  T_{\overline{g}} $ is bounded on $F_{\alpha} ^p $.  Since  $\text{span} \{k_\eta : \eta \in \mathbb{C}^n\}$ is dense in $F_\alpha ^p$ and $F_\alpha ^q$ (see \cite{JPR}), we have that \begin{align} T_f  T_{\overline{g}}  k_z   =  \overline{g(z)} f k_z. \nonumber \end{align} Moreover, it is easy to see that $|f(z)| \leq    C \left(\widetilde{|f|^p}^{\left(\frac{\alpha p}{2}\right)} (z) \right)^\frac{1}{p} $ for some $C > 0$ independent of $f$ and $z$, so that \begin{align} \underset{z \in \mathbb{C}^n}{\sup}  \, |f(z) g(z)| & \leq  \underset{z \in \mathbb{C}^n}{\sup}  \, |g(z)| \left(\widetilde{|f|^p}^{\left(\frac{\alpha p}{2}\right)} (z) \right)^\frac{1}{p} \nonumber \\ & = \underset{z \in \mathbb{C}^n}{\sup}  \, \|T_f  T_{\overline{g}}  k_z  \|_{F_\alpha ^p} \tag{3.6}  \end{align}  which implies that $fg$ is identically a constant since $\|k_z  \|_{F_\alpha^p} = 1$. Also, it is easy to see that either $f \equiv 0$ or $g \equiv 0$ if either $f$ or $g$ vanishes anywhere on $\mathbb{C}^n$.  Thus, assume that both $f$ and $g$ never vanish on $\mathbb{C}^n$.
\medskip
\noindent  Since $(T_f  T_{\overline{g}}  )^*  =  T_{g}  T_{\overline{f}} $ is bounded on $\left(F_\alpha ^p\right) ^* = F_\alpha ^q$ (again see \cite{JPR}), we also have that \begin{align} \underset{z \in \mathbb{C}^n}{\sup} \, |f(z)| \left(\widetilde{|g|^{q}}^{\left(\frac{\alpha q}{2}\right)} (z) \right)^\frac{1}{q} & = \underset{z \in \mathbb{C}^n}{\sup}  \, \|T_{g}  T_{\overline{f}}  k_z  \|_{F_\alpha ^q} \nonumber \tag{3.7} \end{align}

\noindent Combining $(3.6)$  and $(3.7)$ now gives us that \begin{align} \underset{z \in \mathbb{C}^n}{\sup}  \left(\widetilde{|f|^p}^{\left(\frac{\alpha p}{2}\right)} (z) \right)^\frac{1}{p} \left(\widetilde{|g|^{q}}^{\left(\frac{\alpha q}{2}\right)} (z) \right)^\frac{1}{q} < \infty.  \nonumber \end{align}

\medskip
Now we prove that $(b) \Longrightarrow (a)$. If $(b)$ is true, then again $fg$ is identically constant, and if either $f$ or $g$ vanish anywhere on $\mathbb{C}^n$, then one of these functions is identically zero.  Moreover, if $(b)$ is true and both $f$ and $g$ never vanish on $\mathbb{C}^n$, then it is easy to see that $|f|^p \in \text{A}_p ^\text{restricted}$, which means that $T_f T_{\overline{f^{-1}}}$ (and also $T_f T_{\overline{g}}$) is bounded on $F_\alpha ^p$.

  \medskip
   To finish the proof, we show that if $f$ is an entire function with $|f|^p \in \text{A}_p ^{\text{restricted}}$ and $f^{-1}$ is entire, then there exists costants $C_1, C_2, C_3, C_4$ where \begin{equation*} C_1 e^{C_2 |z|} \leq |f(z)| \leq C_3 e^{C_4 |z|} \end{equation*} for any $z \in \mathbb{C}^n$. A simple argument using the Weierstrass factorization theorem in one dimension then shows that $f = e^{P}$ for a linear polynomial $P$ (see \cite{CPZ} for more details.)   To see this, first note that by essentially the definition of BO, we have \begin{equation*} |g(z)| \leq A + B|z| \end{equation*} for some constants $A, B \geq 0$ if $g \in \text{BO}$.  Now since $\log |f| \in \text{BMO}^1$ and is subharmonic, we have that \begin{equation*} \log |f| (z) \leq (\log |f|)_{B(z, 1)} \leq C_1 + C_2 |z| \end{equation*} since $(\log |f|)_{B(z, 1)}  \in \text{BO}.$  Applying the same reasoning to $|f|^{- \frac{p}{p - 1}}$ completes the proof.
\end{proof}

\medskip
\begin{remark}  Using ideas from the proof of Theorem $5.1$ in Section $5$, it is not difficult to see that the following are equivalent for any measurable $f$ on $\mathbb{C}^n:$  \begin{enumerate}
\item[(a)] $ \underset{z \in \mathbb{C}^n}{\sup}  \left(\widetilde{|f|^p}^{(\alpha)} (z) \right)^\frac{1}{p} \left(\widetilde{|f|^{-q}}^{(\beta)} (z) \right)^\frac{1}{q} < C_{\alpha, \beta}$  for some $\alpha, \beta > 0$.
\item[(b)]  $\underset{z \in \mathbb{C}^n}{\sup}  \left(\widetilde{|f|^p}^{(\alpha)} (z) \right)^\frac{1}{p} \left(\widetilde{|f|^{-q}}^{(\beta)} (z) \right)^\frac{1}{q} < C_{\alpha, \beta} $ for all $\alpha, \beta > 0$.
\item[(c)] $w = |f|^p$ belongs to A${}_p ^{\text{restricted}}$.
\end{enumerate}
\end{remark}

\noindent Finally note that the following is a direct consequence of Theorem $3.1$:
\begin{corollary} Let $ f$ be any measurable function on $\mathbb{C}^n$ where $f \not\equiv 0$ a.e. and where $w = |f|^p$ in is A${}_p ^\text{restricted}$.  Then $T_f T_{\overline{f^{-1}}}$ is bounded on $\mathcal{L}_\alpha ^p$ (and in particular, bounded on $F_\alpha ^p$) for any $\alpha > 0 $. Also, the same statement holds for $T_f T_{f^{-1}}$. \end{corollary}

\section{Classes of Weights}

In this section, we will analyze  the classes of weights relevant to the results of the previous sections.

\medskip
\noindent First, for $p > 1$, define the invariant A${}_p$ class (which will be denoted by A${}_p ^\text{inv.}$) to be the class of all weights $w$ on $\partial \mathbb{B}_n$ such that \begin{align} \underset{ z \in \mathbb{B}_n}{\sup} \ \{\widehat{w}(z) \}   \{\widehat{w^{-\frac{1}{{p - 1}}}} (z)\}^{p - 1}  < \infty. \nonumber \end{align} Note that by definition,  A${}_p ^\text{inv.}$  is M\"{o}bius invariant. For a definition and discussion of A${}_\infty ^\text{inv.}$ weights on $\partial\mathbb{D}$, see \cite{TVZ} and \cite{W}.

\medskip
  For $p = 2$, it is not difficult to show that A${}_2$ and A${}_2 ^\text{inv.}$ coincide.  However, for general $p > 1$, A${}_p ^\text{inv.}$ is strictly larger than  A${}_p$  (see \cite{ TVZ, W} for examples.) Also, for a discussion of A${}_p ^\text{inv.}$ weights on $\mathbb{R}$ for $ 1 < p < \infty$, see \cite{H}.

\medskip
 With this in mind, one can similarly define B${}_{p, \gamma} ^\text{inv.} $ to be the class of all weights $w$ on $\mathbb{B}_n$ where  \begin{align} \underset{ z \in \mathbb{B}_n}{\sup} \ \{B_\gamma (w) (z) \}   \{B_\gamma(w^{- \frac{1}{{p - 1}}}) (z)\}^{p - 1}  < \infty. \nonumber \end{align} Note that B${}_{p, \gamma} ^\text{inv.}$ is also M\"{o}bius invariant

\medskip
  \noindent We can also describe B${}_{p, \gamma}  $ in terms of the Berezin transform.  In particular, we have: \begin{proposition} A weight $w$ on $\mathbb{B}_n$ is in B${}_{p, \gamma}$ if and only if \begin{align} \|w \|_{\text{B}_{p, \gamma} ^\text{Ber.}} =  \underset{z \in \mathbb{B}_n}{\sup}  \ \{B_\gamma (w |k_z ^{p - 2}| )  (z) \}  \{B_\gamma (w^{- \frac{1}{p - 1}} |k_z ^{q - 2}|) (z) \}^{p - 1}  < \infty.  \nonumber\end{align} In particular, there exists a constant $ C$ independent of $w$ where \begin{align}  \frac{1}{C} \|w\|_{\text{B}_{p, \gamma} }  \leq  \|w\|_{\text{B}_{p, \gamma} ^\text{Ber.}}   \leq C \|w\|_{\text{B}_{p, \gamma} } ^{\max\{p + 1, q +1 \}}. \nonumber \end{align} \end{proposition}
 \noindent Note that this proposition tells us that B${}_{p, \gamma} ^\text{inv.} = {\text{B}}{}_{p, \gamma} $ when $p = 2$.

\medskip
If we define $w_\zeta (z) = (1 - |z|^2)^\zeta$ for $\zeta \in \mathbb{R}$, then a messy but elementary application of the Rudin-Forelli estimates (see \cite{Z2}) gives us the following two propositions:

\begin{proposition} $w_\zeta \in \text{B}{}_{p, \gamma}$ if and only if
$ -1 - \gamma < \zeta < (1 + \gamma)(p - 1).$
\end{proposition}

\begin{proposition} $w_\zeta \in \text{B}{}_{p, \gamma} ^{\text{inv.}}$ if and only if
\begin{enumerate}
\item[(1)] $ -1 - \gamma < \zeta < (1 + \gamma)(p - 1)$, and
\item[(2)] $ -(p -1)(n + 1 +  \gamma) < \zeta  < n + 1 + \gamma.$
\end{enumerate}
\end{proposition}

\noindent These two propositions tell us that the classes B${}_{p, \gamma} ^{\text{inv.}}$ and $B{}_{p, \gamma}$ do not coincide when either $p > 2 + \frac{n}{1 + \gamma}$ or $p < 1 + \frac{1 + \gamma}{n + 1 + \gamma}$.  However, it is unlikely that B${}_{p, \gamma} ^{\text{inv.}}$ and $B{}_{p, \gamma}$ coincide for any $p > 1, n \geq 1$, and $\gamma > -1$.

\medskip
Also, we have the following analog of Proposition $4.1$ for $\partial \mathbb{B}_n$: \begin{proposition} A weight $w$ on $\partial \mathbb{B}_n$ is in A${}_{p}$ if and only if $w$ satisfies \begin{align}  \|w\|_{\text{A}_p ^\text{Poi.}} = \underset{z \in \mathbb{B}_n}{\sup}  \ \{\widehat{w | k_z |^{p- 2}}  (z) \} \{\widehat{w^{- \frac{1}{p - 1}} |k_z| ^{q - 2}} (z) \}^{p - 1}  < \infty. \nonumber  \end{align}  In fact, there exists a constant $ C$ independent of $w$ where \begin{align} \frac{1}{C} \|w\|_{\text{A}_p } \leq  \|w\|_{\text{A}_p ^\text{Poi.}} \leq C \|w\|_{\text{A}_p} ^{\max\{p + 1, q +1 \}}. \nonumber \end{align}   \end{proposition} \noindent  Here, $k_z (w) = \frac{(1 - |z|^2)^{n/2}}{(1 - w \cdot z )^n}$ is the normalized reproducing kernel of $H^2(\partial \mathbb{B}_n)$.

\medskip
   We will defer the proof of Propositions $4.1$ and $4.4$ until the last section since the proof uses ideas found there.  It should be noted that Propositions $4.1$ and $4.4$ are interestingly \textit{not} true in the $\mathbb{R}^n$ setting when $n \geq 2$. In particular, if $w(x) = |x|^\alpha$, then $w \in \text{A}_2$ if and only if $|\alpha| < n $, whereas the integrals in the expression for the $\text{A}_2 ^\text{poi}$ characteristic of $w$ diverge if $\alpha \geq 1$  (see \cite{St}).

\medskip
When $p = 2$, Proposition $4.4$ was proven to be sharp in \cite{H} for $\mathbb{R}.$  In particular, if $-1 < \alpha < 1/2$ and if \begin{align} w =
\left\{\begin{array}{ll} 1
 & \textrm{ for } x \in [0, 1)^c  \\ (1 - \alpha)^n & \textrm{ for }  x \in (1/2^{n + 1}, 1/2^n]
\end{array} \right. \tag{4.1} \end{align} then $\|w\|_{\text{A}_2} \approx (1 - 2\alpha)^{-1}$, while $\|w\|_{\text{A}_2 ^\text{inv.}} \approx (1 - 2\alpha)^{-2}$.  Since virtually the same example produces the same conclusion on $\partial \mathbb{D}$, it would be interesting to know if some example similar to $(4.1)$ can be cooked up for the unit disk or the unit ball.

\medskip

Note that proposition $4.2$ immediately gives us that one can \textit{not} define the class B${}_{p, \gamma}$ in terms of Bergman balls of a fixed radius.  In particular, note that $w_\zeta (z) = (1 - |z|^2)^\zeta$ for any $\zeta \in \mathbb{R}$ satisfies \begin{align} \underset{z \in \mathbb{B}_n}{\sup} \  \left(\frac{1}{v_\gamma (D(z, r))} \int_{D(z, r)} w_\zeta \, dv_\gamma \right) \left(\frac{1}{v_\gamma (D(z, r))} \int_{D(z, r)} w_\zeta  ^{ - \frac{1}{p - 1}} \, dv_\gamma \right)^{p - 1} < C_r  \nonumber \end{align} for some $C_r \geq 1$, where here $D(z, r) \subseteq \mathbb{B}_n$ is a ball with respect to the Bergman metric with center $z$ and radius $r$.

\medskip
 It should be remarked that the Muckenhoupt A${}_p$ class on $\mathbb{R}^n$ coincides with the class of all weights $w$ on $\mathbb{R}^n$ such that \begin{align} \|w\|_{\text{A}_p ^{\text{heat}}} = \underset{(x, \alpha) \in \mathbb{R}^n \times \mathbb{R}_+}{\sup} \  \left(\widetilde{w}^{(\alpha)} (x) \right) \left(\widetilde{ w^{- \frac{1}{p - 1}}} ^{(\alpha)}(x) \right)^{p - 1} < \infty \nonumber \end{align}  (this was proven in \cite{PV} for $n = 1$, but the proof can easily be extended to the $n > 1$ case.) Moreover, the characteristics defined by the corresponding supremums are equivalent.

\medskip
\noindent On the other hand, an argument that is similar to (but easier than) the proof of Theorem $3.1$ tells us that A${}_p ^\text{restricted}$ coincides with the class of all weights $w$ on $\mathbb{R}^n$ where for each $\alpha, \beta > 0$, there is some $C_{\alpha, \beta} < \infty$ such that \begin{align} \underset{x \in \mathbb{R}^n}{\sup} \ \left(\widetilde{w}^{(\alpha)} (x) \right) \left( \widetilde{w^{- \frac{1}{p - 1}}}^{(\beta)} (x) \right)^{p - 1} < C_{\alpha, \beta} \tag{4.2}. \end{align} Unfortunately, the argument gives no relationship between $(4.2)$ for fixed $\alpha, \beta$ and the A${}_{p, r}$  characteristic of a weight for a fixed $r$, though trivially there exists $C_{r, \alpha, \beta}$ where \begin{align} \|w\|_{\text{A}{}_{p, r}} \leq C_{r, \alpha, \beta} \underset{x \in \mathbb{R}^n}{\sup} \ \left(\widetilde{w}^{(\alpha)} (x) \right) \left( \widetilde{w^{- \frac{1}{p - 1}}}^{(\beta)} (x) \right)^{p - 1}. \nonumber \end{align}

\medskip We will end our discussion of A$_p$ and B${}_{p, \gamma}$ weights by comparing one last property of A${}_p$ and B${}_{p, \gamma}$ weights.  Recall that Coifmann and Fefferman proved in \cite{CF} that \begin{align} \text{A}_p = \bigcup_{1 < q < p} {\text{A}}_{q} \nonumber \end{align} if $ p > 1$.   Note that one side of this equality holds trivially by H\"{o}lder's inequality.  Using Lemma $5.8$ in Section $5$, it is not difficult to see that \begin{align} \mathcal{A} \cap  {\text{B}}_{p, \gamma}  \subseteq \bigcup_{1 < q < p} {\text{B}}_{q, \gamma} \tag{4.3} \end{align} where $\mathcal{A} $ is the collection of all $|f|$ such that $f$ is analytic on $\mathbb{D}$ with no zeros in $\mathbb{D}$.

 When $\gamma = 0$, Borichev generalized $(4.3)$ and proved (among other things) that  \begin{align} \mathcal{E} \mathcal{S}   \cap  \text{B}_{p, \gamma}  \subseteq \bigcup_{1 < q < p} \text{B}_{q, \gamma} \nonumber \end{align} where $\mathcal{E} \mathcal{S} $ is the class of all functions $e^u$ for $u$ subharmonic on $\mathbb{D}$ (see \cite{B}).  Furthermore,  it was shown in \cite{B} that if $\mathcal{S}$ is the class of non-negative subharmonic functions on $\mathbb{D}$, then \begin{align} \mathcal{S} \cap  \text{B}_{p, \gamma}  \nsubseteq \bigcup_{1 < q < p} \text{B}_{q, \gamma}. \nonumber  \end{align}   Given these results, it would be interesting to know if the results in \cite{B}, can be extended to general $\gamma > -1$ and $n > 1$, or if $(4.3)$ is true for $n > 1$ .

\section{ A ``reverse H\"{o}lder inequality" on $\mathbb{D}$.}

In this last section, we will provide a proof of Theorem $2.1$ for the disk $\mathbb{D}$ by extending the ideas of \cite{SZ1, SZ2} from the $p = 2$ case to the general $p > 1$ case. In particular, we will prove the following ``reverse H\"{o}lder inequality:" \begin{theorem} Let $f \in L_a ^p (\mathbb{D}, dA_\gamma)$ and $f^{-1} \in L_a ^q (\mathbb{D}, dA_\gamma)$ satisfy \begin{align} \underset{z \in \mathbb{D}}{\sup} \ \left\{B_\gamma \left(|f k_z^{1 - 2/p}|^p \right)(z) \right\}^\frac{1}{p}  \left\{B_\gamma \left(|f^{-1} k_z^{1 - 2/q}|^q \right)(z) \right\}^\frac{1}{q} < \infty. \tag{5.1} \end{align} Then there exists $\epsilon > 0$ such that \begin{align} \underset{z \in \mathbb{D}}{\sup} \ \left\{B_\gamma \left(|f k_z^{1 - 2/p}|^{p + \epsilon}  \right)(z) \right\}^\frac{1}{p + \epsilon}  \left\{B_\gamma \left(|f^{-1} k_z^{1 - 2/q}|^{q + \epsilon} \right)(z) \right\}^\frac{1}{q + \epsilon} < \infty. \tag{5.2} \end{align}\end{theorem}

\noindent
Once this is proved, Theorem $1.2$ of \cite{M} will give us that $T_f T_{\overline{f^{-1}}}$ is bounded on $L_a ^p(\mathbb{D}, dA_\gamma)$.  Easy arguments from Section $2$ will then complete the proof of Theorem $2.1$ for $n = 1$.

\medskip
 When $p = 2$, condition $(5.2)$ is M\"{o}bius invariant, so that it is only necessary to prove that $(5.1)$ implies $(5.2)$ when $z = 0$ in $(5.2)$ (which was done in \cite{SZ1, SZ2}.) In other words, it is proven in \cite{SZ1, SZ2} that if both $f,  \  f^{-1} \in L_a ^2 (\mathbb{D}, dA_\gamma)$ satisfy \begin{align} \underset{z \in \mathbb{D}}{\sup} \ \left\{B_\gamma \left(|f|^2 \right)(z) \right\}^\frac{1}{2}  \left\{B_\gamma \left(|f|^{-2} \right)(z) \right\}^\frac{1}{2} < \infty, \nonumber \end{align} then there exists $\epsilon > 0$ and $C > 0$ such that \begin{align}\left(\int_{\mathbb{D}} |f|^{2 + \epsilon}  \, dA_\gamma \right)^\frac{1}{2 + \epsilon}  \leq C \left(\int_{\mathbb{D}} |f|^{2 }  \, dA_\gamma \right)^\frac{1}{2 }. \tag{5.3}\end{align}  When $p \neq 2$, condition $(5.2)$ is not necessarily  M\"{o}bius invariant, which means that it is not enough to just verify $(5.3)$ (where $p$ replaces $2$.)

 \medskip To prove Theorem $5.1$, we will decompose $\mathbb{D}$ into convenient Carleson squares using the ``Bergman tree" of \cite{ARS}. We will then run a Calderon-Zygmund decomposition on each of these Carleson squares to prove a reverse H\"{o}lder type inequality on each of these Carleson squares that is similar to $(5.3)$. This will allow us to prove that $f$ satisfies an ``A${}_\infty$ type" condition with respect these Carleson squares if $f$ satisfies $(5.1)$.   The decay provided by the normalized Bergman kernel, combined with this ``A${}_\infty$ type" condition, will then allow us to prove Theorem $5.1$

  We will now go through the details of the proof of Theorem $5.1$. In what follows, we will use the notation $A \approx B$ for two quantities $A$ and $B$ if there exists $C > 0$ depending only on $\gamma, n, $ and $p$ where \begin{align} \frac{1}{C} A \leq B \leq C A. \nonumber \end{align} The notation $A \lesssim B$ and $A \gtrsim B$ will have similar meanings.  For any $0 < h \leq 1$ and $0 \leq \theta < 2\pi$, let $S_{h, \theta} \subseteq \mathbb{D}$ denote the Carleson square defined by\begin{align} S_{h, \theta} = \{re^{i t} : 1 - h \leq r < 1, \ \theta \leq t < \theta + h\}  \nonumber \end{align} and let \begin{align} T_{h, \theta} = \{re^{i t} : 1 - h \leq r < 1 - \frac{h}{2}, \ \theta \leq t < \theta + h\} \nonumber  \end{align} denote the ``bottom half" of the Carleson square $S_{h, \theta}$. Here we will only be interested in Carleson and bottom half Carleson squares of the form $S_{h, \theta}$ where $h = 2^{-n}$ and $\theta = 2\pi (k 2^{-n})$ for $n = 0, 1, 2, \ldots$ and $k = 0, 1, \ldots, 2^n - 1$.

\medskip
Let us now introduce the ``Bergmann tree" of Arcozzi, Rochberg, and Sawyer for $\mathbb{D}$ from \cite{ARS}.  Let $\mathcal{D}$ be the index set defined by \begin{align} \mathcal{D} = \{(n, k) : n = 0, 1, 2, \ldots \text{ and } k = 0, 1, \ldots, 2^n - 1\}. \nonumber \end{align} We call $o = (0, 1)$ the root of $\mathcal{D}$.  We give $\mathcal{D}$ a partial ordering by declaring $\eta \leq \beta$ if $S_\beta \subseteq S_\eta$, and call $\mathcal{D}$ with this partial ordering the Bergman tree.  Note that this partial ordering means that $S_o \leq S_\beta$ for every $\beta \in \mathcal{D}$. Also, we will let $c_\beta$ denote the center (radially and angularly) of $T_\beta$ and let $d(\beta) = n$ if $\beta = (n, k)$. Moreover, if $\beta \leq \beta'$ with $d(\beta) = d(\beta') - 1$ then we say $\beta'$ is a child of $\beta$.  Cleary each $\beta \in \mathcal{D}$ has only two children. Note that by definition we have that \begin{align} S_\eta = \bigcup_{\beta \geq \eta} T_\beta. \nonumber \end{align}
\medskip

If $z, w \in \mathbb{D}$ where $z = r e^{i\theta}, \  w = s e^{i \vartheta}$, and $0 \leq \theta, \vartheta < 2\pi$, then it is easy to see that \begin{align} |1 - z \overline{w}|^2 = (1 - rs)^2 + 4rs \sin^2 \left(\frac{\theta - \vartheta}{2}\right). \tag{5.4} \end{align} Thus, there exists $R> 0$ independent of $\beta \in \mathcal{D}$ such that \begin{align} D(c_\beta, 1/R) \subseteq T_\beta \subseteq D(c_\beta, R) \tag{5.5} \end{align} where $D(z, r)$ is a Bergman disk of radius $r$ and center $z$.  Also, it is not difficult to see that \begin{align} A_\gamma (T_\beta) \approx A_\gamma (S_\beta) \approx 2^{- d(\beta) (2 + \gamma)}
\nonumber \end{align} for each $\beta \in \mathcal{D}$.

\medskip
Given any $S_\beta$ with $\beta \in \mathcal{D}$, we can form dyadic partitions of $S_\beta$ by dyadically bisecting $S_\beta$ in the angular and radial direction.  Any subset $Q \subset S_\beta$ formed in this way will be called a dyadic subrectangle of $S_\beta$.  Note that since $\mathbb{D} = S_o$, the ``dyadic rectangles" of \cite{SZ1, SZ2} are dyadic subrectangles of $\mathbb{D}$ according to our definition. In particular, any dyadic subrectangle of $\mathbb{D}$ can be written in the form \begin{align}Q_{n, m, k} = \{re^{i\theta} : (m - 1) 2^{-n}  \leq r < m 2^{-n} \text{ and } (k - 1) 2^{-n + 1} \pi \leq \theta < k 2^{-n + 1} \pi \} \nonumber \end{align} where $k, m, $ and $n$ are positive integers such that $m, k \leq 2^n$.   Also, the center of $Q = Q_{n, m, k} $ is the point $z_Q = (m - \frac{1}{2}) 2^{-n} e^{i \vartheta}$ with $\vartheta = (k - \frac{1}{2})  2^{1-n} \pi$.  Throughout this section we will use $z_Q$ to denote the center (angularly and radially) of a dyadic subrectangle of $\mathbb{D}$, whereas $c_\beta$ will denote the center of $T_\beta$ for $\beta \in \mathcal{D}$.

\medskip
\begin{lemma} Let $f \in L_a^p (\mathbb{D}, dA_\gamma)$ satisfy $(5.1)$ and let $R > 0$. Then there exists $C_R > 0$ such that \begin{align} \frac{1}{C_R} \leq \frac{|f(z)|}{|f(w)|} \leq C_R \nonumber \end{align} whenever $z \in D(w, R)$.  \end{lemma}

\begin{proof} The proof is very similar to the proof of Lemma $4.3$ in \cite{SZ2}, though we include it for the sake of completeness. According to lemma $4.30$ in \cite{Z1}, there exists $C > 0$ depending on $n, p, R,$ and $\gamma$ such that \begin{align} \frac{1}{C }  (1 - |w|^2)^{\left(\frac{2}{p} - 1\right) \left(\frac{2 + \gamma}{2}\right) } \leq |k_w ^{1 - \frac{2}{p}} (z)| \leq { C} (1 - |w|^2)^{\left(\frac{2}{p} - 1\right) \left(\frac{2 + \gamma}{2}\right) }\nonumber \end{align} whenever $z \in D(w, R)$,.

\medskip
 For $z \in D(w, R)$, let $z = \varphi_w (u)$ with $u \in D(0, R)$. Then we have that \begin{align} |f(z)| & \leq   C (1 - |w|^2)^{\left(  1 - \frac{2}{p}   \right)\left(\frac{2 + \gamma}{2}\right)} |f(\varphi_w (u))| |k_w ^{1 - \frac{2}{p}} (\varphi_w (u))|  \nonumber \\ & \leq  C  (1 - |w|^2)^{\left( 1 - \frac{2}{p}  \right)\left(\frac{2 + \gamma}{2}\right)}  \{ B_\gamma ( |f k_w ^{1 - \frac{2}{p}} |^p  ) (w) \}^\frac{1}{p}. \nonumber\end{align}

\medskip
\noindent Similarly, for $f^{-1}$ we have that \begin{align} \frac{1}{|f(w)|} \leq C (1 - |w|^2)^{\left( 1 - \frac{2}{q}  \right)\left(\frac{2 + \gamma}{2}\right)}  \{ B_\gamma ( |f^{-1} k_w ^{1 - \frac{2}{q}} |^q  ) (w) \}^\frac{1}{q} \nonumber \end{align} which means that \begin{align} \frac{|f(z)|}{|f(w)|} \leq C   \{ B_\gamma ( |f k_w ^{1 - \frac{2}{p}} |^p  ) (w) \}^\frac{1}{p} \{ B_\gamma ( |f^{-1} k_w ^{1 - \frac{2}{q}} |^q  ) (w) \}^\frac{1}{q} \leq C \nonumber \end{align} where here $C$ depends on $R$ and the the supremum in $(5.1)$.  Replacing $f$ by $f^{-1}$ and $p$ with $q$ in the above argument now completes the proof.  \end{proof}

\medskip
The following two results were proven in \cite{SZ2}.
\begin{proposition} For every dyadic subrectangle $Q$ of $\mathbb{D}$ and every $z \in Q$, we have that \begin{align}
|k_{z_Q} (z)|^2 \gtrsim \frac{1}{(1 - |z_Q|^2)^{2 + \gamma}} \nonumber \end{align}  \end{proposition}

\begin{proposition} There exists $R > 0$ such that $Q \subseteq D(z_Q, R)$ for every dyadic subrectangle $Q$ of $\mathbb{D}$ that has positive distance to $\partial \mathbb{D}$.  \end{proposition}

\begin{lemma} Let $f \in L_a^p (\mathbb{D}, dA_\gamma)$ satisfy $(5.1)$ and let $w = |f|^p$.  Then for each $\beta \in \mathcal{D}$ and each dyadic subrectangle $Q$ of $S_\beta$, we have that \begin{align} \left(\frac{1}{A_\gamma(Q)} \int_Q w \,  dA_\gamma \right)\ \left(\frac{1}{A_\gamma(Q)} \int_Q w ^{- \frac{1}{p - 1}}\,  dA_\gamma \right)^{p - 1}  \leq C \tag{5.6} \end{align} where $C$ is independent of $\beta$ and $Q$. \end{lemma}

\begin{proof} Clearly it is enough to show that there exists $C > 0$ independent of $\beta$ and $Q$ where \begin{align} \left(\frac{1}{A_\gamma(Q)} \int_Q |f|^p dA_\gamma \right)^\frac{1}{p} \left(\frac{1}{A_\gamma(Q)} \int_Q |f|^{-q} dA_\gamma \right)^\frac{1}{q }  \leq C.  \nonumber \end{align}

\medskip
\noindent First assume that $\beta = o$, so that $S_\beta = \mathbb{D}.$  If $Q = \mathbb{D},$ then this follows immediately from $(5.1)$.  If $d(Q, \partial \mathbb{D}) > 0$ then the result immediately follows from Proposition $5.4$ and Lemma $5.2$.  If $d(Q, \partial \mathbb{D}) = 0$ then the Lemma follows from Proposition $5.3$ and the fact that $A_\gamma(Q) =  2^{3 + 2\gamma} |z_Q| ^{1 + \gamma} (1 - |z_Q|)^{2 + \gamma}$ (see \cite{SZ2}.)

 Now assume that $\beta \neq o$.  Note that if we dyadically quadrisect $S_\beta$ any number of times, then an easy induction shows that we either obtain one of three types of sets: $S_{\beta'}$ where $\beta' \geq \beta$, the left (or right) angular half of some $T_{\beta'}$, or repeated quadrisection of the left (or right) angular half of some $T_{\beta'}$.  In particular, this tells us that any dyadic subrectangle $Q$ of $S_\beta$ is either $S_{\beta'}$ for some $\beta' \geq \beta$ or is contained in the hyperbolic disk $D(c_{\beta'}, R)$ where $\beta' \geq \beta$ and $R$ is the constant from $(5.5)$.

\medskip
In the latter case, the Lemma follows immediately from Lemma $5.2$.  To finish the proof, we will show that Lemma $5.5$ is true for each $S_\beta$. If $z \in S_\beta$   with $z = re^{i\theta}$  where $0 \leq \theta < 2\pi$, then by the definition of $S_\beta$ we have that  $| \theta - \vartheta| \leq 2^{- d(\beta)}$ where $c_\beta = s e^{i \vartheta}$ with $0 \leq \vartheta < 2\pi$. Thus, since $(1 - |c_\beta|^2) \approx 2^{-d(\beta)}$, we have from $(5.4)$ that \begin{align} |k_{c_{\beta}} (z)|  = \frac{(1 - |c_\beta|^2)^\frac{2 + \gamma}{2} }{|1 - z\overline{c_\beta} |^{2 + \gamma}}  \gtrsim   \frac{1}{ A_\gamma(S_\beta) ^{1/2}} \nonumber \end{align}
 \noindent which tells us that \begin{align} \{B_\gamma(|f k_{c_\beta} ^{1 - \frac{2}{p}}|^p) (c_\beta)\}^\frac{1}{p}  & = \left(\int_\mathbb{D} |f k_{c_\beta}| ^p \, dA_\gamma \right)^\frac{1}{p} \nonumber \\ & \geq \left(\int_{S_\beta} |f k_{c_\beta}| ^p \, dA_\gamma \right)^\frac{1}{p} \nonumber \\ & \gtrsim  \frac{1}{ A_\gamma(S_\beta) ^{1/2}} \left(\int_{S_\beta} |f | ^p \, dA_\gamma \right)^\frac{1}{p}. \nonumber \end{align}

\noindent Switching $f$ with $\frac{1}{f}$, and switching $p$ with $q$, now completes the proof. \end{proof}

\medskip
The proof of the following is a standard application of Lemma $5.5$ (and is very similar to the proof of Lemma $4.6$ of \cite{SZ2}). The proof will therefore be omitted. Note that for the rest of this section $\gamma > -1$ will be fixed and for a measurable set $E \subseteq \mathbb{D}$ we will use the notation $w(E) = \int_E w \, dA_\gamma$.

\begin{lemma} Let $f \in L_a^p (\mathbb{D}, dA_\gamma)$ satisfy $(5.1)$ and let $C_1$ be the constant in Lemma $5.5$. If $w =  |f|^p $ and if $\delta = 1 - \frac{1}{2^p C_1} $, then we have  \begin{align} w(E) \leq \delta w(Q) \nonumber \end{align} whenever $E$ is a subset of a dyadic subrectangle $Q$ of any $S_{\beta}$ where $A_\gamma (E) \leq \frac{1}{2} A_\gamma (Q)$. \end{lemma}

Now, suppose that we have a dyadic subrectangle $Q$ of $S_\beta$ for some $\beta \in \mathcal{D}$.  If $Q$ is formed from $k \geq 1$ repeated dyadic quadrisections of $S_\beta$, then we define the double $2Q$ of $Q$ to be the unique dyadic subrectangle of $S_\beta$ formed by $k - 1$ repeated dyadic quadrisections of $S_\beta$ that also contains $Q$.
We will now establish a doubling property that extends Proposition $4.9$ of \cite{SZ2}.

\begin{lemma}
For any $\beta \in \mathcal{D}$ and any dyadic subrectangle $Q \subsetneqq S_\beta$, we have that $A_\gamma(2Q) \lesssim A_\gamma(Q)$. \nonumber \end{lemma}

\begin{proof} If $Q $ is a dyadic subrectangle of $\mathbb{D}$, then this was proven in Proposition $4.9$ of \cite{SZ2}, so assume that $Q$ is a dyadic subrectangle of $S_\beta$ with $d(\beta) \geq 1$.

\noindent
As stated in the proof of Lemma $5.5$, repeated quadrisection of $S_\beta$ gives us one of the following three sets: $S_{\beta'}$ where $\beta' \geq \beta$, the left (or right) angular half of $T_{\beta'}$, or the repeated quadrisection of the left (or right) angular half of $T_{\beta'}$. However, since $A_\gamma (S_\beta) \approx A_\gamma(T_\beta) \approx 2^{- d(\beta)(2 + \gamma)}$, it is easy to see that $A_\gamma (2Q) \leq C A_\gamma (Q)$ for either of these cases, where $C > 0$ is independent of $Q$.
\end{proof}

\begin{lemma} Let $f \in L_a ^p(\mathbb{D}, dA_\gamma)$ satisfy $(5.1)$. Also, let $\widetilde{C} > 0$ be the constant in Lemma $5.7$ and let $\delta$ be the constant from Lemma $5.6$. If $\beta \in \mathcal{D},$ then for any dyadic subrectangle $Q$ of $S_\beta$ (including $S_\beta$ itself), we have that \begin{align} \left(\frac{1}{A_\gamma(Q)} \int_{Q} w^{1 + \epsilon} \, dA_\gamma \right) ^ \frac{1}{1 +  \epsilon} \leq \left(1 + \frac{(2\widetilde{C})^\epsilon}{1 - (2 \widetilde{C})^\epsilon \delta}\right) ^\frac{1}{1+ \epsilon} \frac{1}{A_\gamma(Q)} \int_{Q} w \, dA_\gamma  \nonumber \end{align} whenever $(2 \widetilde{C})^\epsilon \delta < 1$. \end{lemma}

\begin{proof} Using Lemmas $5.6$ and $5.7$, the proof is identical to the proof of Theorem $7.4$ in \cite{D}. \end{proof}

\begin{lemma} Let $f \in L_a ^p (\mathbb{D}, dA_\gamma)$ satisfy $(5.1)$ and let $w = |f|^p$.  Then for any $\beta \in \mathcal{D}$, any $E \subset S_\beta$, and  small enough $\epsilon$, we have  \begin{align} \int_{E} w^{ 1 + \epsilon} \, dA_\gamma \leq C \left( \int_{S_\beta} w ^{ 1 + \epsilon} \, dA_\gamma \right) \left( \frac{A_\gamma(E)}{A_\gamma(S_\beta)}\right)^ {\frac{\epsilon}{1 + \epsilon}} \nonumber \end{align}  where $C$ is independent of $E$ and $\beta$.    \end{lemma}

\begin{proof}
The proof is similar to the proof of Corollary $7.6$ of \cite{D}, but requires a somewhat careful tracking of the constants involved.  Let $\beta \in \mathcal{D}$ and let $Q$ be any dyadic subrectangle of $S_\beta$.  By Lemma $5.8$, \begin{align} \left(\frac{1}{A_\gamma(Q)} \int_{Q} w^{1 + \epsilon_1} \, dA_\gamma \right) ^ \frac{1}{1 +  \epsilon_1} \leq \left(1 + \frac{(2\tilde{C})^{\epsilon_1}}{1 - (2 \widetilde{C})^{\epsilon_1} \delta}\right)^\frac{1}{1 + \epsilon_1}   \frac{1}{A_\gamma(Q)} \int_{Q} w \, dA_\gamma  \tag{5.7}\end{align} whenever $(2 \widetilde{C})^{\epsilon_1} \delta < 1$ where $\delta = 1 - \frac{1}{2^p C_1}$ and $C_1$ is the constant in Lemma $5.5$.

\medskip
  Similarly, since $w^{- \frac{1}{p - 1}} $ satisfies the conclusion of Lemma $5.5$ with A${}_q$ characteristic $C_1 ^ {q - 1}$, we have that \begin{align} \left(\frac{1}{A_\gamma(Q)} \int_{Q} w^{- (1 + \epsilon_1) ( \frac{1}{p - 1}) } \, \right. & dA_\gamma \Biggr) ^ \frac{1}{1 +  \epsilon_1} \nonumber \\ &  \leq \left(1 + \frac{(2\tilde{C})^{\epsilon_1}}{1 - (2 \widetilde{C} )^{\epsilon_1} \delta'}\right)^{\frac{1}{1 + \epsilon_1}}   \frac{1}{A_\gamma(Q)} \int_{Q} w^{- \frac{1}{p - 1}} \, dA_\gamma \tag{5.8} \end{align} whenever $(2 \widetilde{C})^{\epsilon_1} \delta' < 1$ where $\delta' = 1 - \frac{1}{2^q C_1 ^{q-1}}$.

\medskip
 Combining $(5.6), (5.7)$, and  $(5.8)$, we have that \begin{align}  \left( \frac{1}{A_\gamma(Q)} \int_{Q} w^{1 + \epsilon_1} \, dA_\gamma \right ) & \left( \frac{1}{A_\gamma(Q)} \int_{Q} w^{-(1 + \epsilon_1)(\frac{1}{p - 1})}  \, dA_\gamma \right)^{p - 1} \nonumber \\ & \leq   C_1 ^{1 + {\epsilon_1}}   \left(1 + \frac{(2\tilde{C})^{\epsilon_1}}{1 - (2 \widetilde{C})^{\epsilon_1} \delta}\right) \left(1 + \frac{(2\tilde{C})^{\epsilon_1}}{1 - (2 \widetilde{C} )^{\epsilon_1} \delta'}\right)^{p-1} \tag{5.9} \end{align} which means that $w^{1 + \epsilon_1}$ satisfies the conclusion of Lemma $5.5$  (for small enough ${\epsilon_1}$) with A${}_p$ characteristic  \begin{align} C_{1, {\epsilon_1}}  = C_1 ^{1 + {\epsilon_1}}   \left(1 + \frac{(2\tilde{C})^{\epsilon_1}}{1 - (2 \widetilde{C})^{\epsilon_1} \delta}\right)  \left(1 + \frac{(2\tilde{C})^{\epsilon_1}}{1 - (2 \widetilde{C} )^{\epsilon_1} \delta'}\right)^{p-1}. \nonumber \end{align}

 Moreover, $(5.9)$ implies that Lemma $5.6$ holds for $w^{1 + \epsilon_1}$ with constant $\delta_{\epsilon_1} = 1 - \frac{1}{2^p C_{1, {\epsilon_1}}}  $, and so another application of Lemma $5.8$ with $Q = S_\beta$ gives us that \begin{align} \left(\frac{1 }{A_\gamma (S_\beta)} \int_{S_\beta}      w ^{(1 + \epsilon_1)(1 + \epsilon_2)} \, \right.  & dA_\gamma   \Biggr) ^\frac{1}{1 + \epsilon_2} \nonumber  \\ & \leq \left(1 + \frac{(2 \widetilde{C} )^{\epsilon_2 }}{1 - (2 \widetilde{C})^{\epsilon_2} \delta_{\epsilon_1}}  \right)^\frac{1}{1 + \epsilon_2} \frac{ 1 }{A_\gamma (S_\beta)} \int_{S_\beta} w^{1 + \epsilon_1}  \, dA_\gamma \tag{5.10}\end{align} so long as $\epsilon_2 > 0$ is chosen small enough to make $(2 \widetilde{C} )^{\epsilon_2} \delta_{\epsilon_1} < 1$.

\medskip
Finally, setting $\epsilon = \epsilon_2 = \epsilon_1  $ where $\epsilon$ is chosen small enough and using $(5.10)$ and H\"{o}lder's inequality, we have \begin{align}   w^{1 + \epsilon} (E)  & = \int_{S_\beta} \chi_E w^{1 + \epsilon} \,  dA_\gamma \nonumber \\ & \leq \left(w^{(1 + \epsilon)(1 + \epsilon)} (S_\beta) \right)^{\frac{1}{1 + \epsilon}} A_\gamma (E) ^{\frac{\epsilon}{1 + \epsilon}} \nonumber \\ &  \leq C  w^{1 + \epsilon} (S_\beta) \left(\frac{A_\gamma(E)}{A_\gamma(S_\beta)}\right)^{\frac{\epsilon}{1 + \epsilon}} \nonumber \end{align} \end{proof}

\medskip
 We may now complete the proof of Theorem $5.1$.  If $\beta \in \mathcal{D}$ with $\beta = (n, k)$, then define $\widetilde{S_\beta}$ to be \begin{align} \widetilde{S_\beta} = S_{(n, k-1)} \cup S_{(n, k)} \cup S_{(n, k + 1)}. \nonumber \end{align} Fix $u \in \mathbb{D}$ and pick $\beta \in \mathcal{D}$ such that $u \in T_\beta$. Because of Lemma $5.8$, we may assume that $d(\beta) \geq 2$. For any $o < \eta \leq \beta$, let $\widetilde{\eta }$ be the parent of $\eta$. Then by $(5.4)$ and the definition of $  \widetilde{S_\eta}$,  we have that \begin{align} \underset{z \in \mathbb{D} \backslash \widetilde{S_{\eta}} }{\sup}   |k_u (z)|^2 \lesssim 2^{-d(\beta) \left(2 + \gamma \right)} 2^{2 d(\eta) (2 + \gamma)} \lesssim \frac{1}{A_\gamma(\widetilde{S_{\widetilde{\eta}}})} 2 ^{- (d(\beta) - d(\eta)) \left(2 + \gamma\right)}. \tag{5.11} \end{align}

\noindent Using $(5.11)$ and the fact that \begin{align} \mathbb{D} = \left(\bigcup_{o < \eta \leq \beta} \widetilde{S_{\widetilde{\eta}}} \backslash \widetilde{S_\eta}\right) \cup \widetilde{S_\beta}, \nonumber \end{align}  we have that \begin{align} \{B_\gamma (|f k_u ^{ 1- \frac{2}{p }}|^{p + \epsilon_1})  (u) & \}^\frac{1}{p + \epsilon_1} =  \left(\int_\mathbb{D} |f k_u ^{ 1- \frac{2}{p }}|^{p + \epsilon_1} |k_u|^2 \, dA_\gamma \right)^\frac{1}{p + \epsilon_1}  \nonumber \\ & \leq \sum_{o < \eta \leq   \beta}  2^{-d(\beta) \left(\frac{2 + \gamma}{2} \right)\left(1 - \frac{2}{p}\right)} 2^{ d(\eta) (2 + \gamma) (1 - \frac{2}{p})} \nonumber \\ & \times 2 ^{- \frac{2 + \gamma}{p + \epsilon_1} (d(\beta) - d(\eta)) } \left(\frac{|f|^{ p + \epsilon_1} (\widetilde{S_{\widetilde{\eta}}})}{A_\gamma( \widetilde{S_{\widetilde{\eta}}})}   \right) ^\frac{1}{p + \epsilon_1}.      \nonumber   \end{align}

\noindent Similarly, we have \begin{align}  \{B_\gamma  (|f^{-1} k_u ^{ 1- \frac{2}{q }}|^{q + \epsilon_2})   (u) \}^\frac{1}{q + \epsilon_2}  & =  \left(\int_\mathbb{D} |f^{-1} k_u ^{ 1- \frac{2}{q }}|^{q + \epsilon_2} |k_u|^2 \, dA_\gamma \right)^\frac{1}{q + \epsilon_2}  \nonumber \\ & \lesssim \sum_{o < \eta ' \leq  \beta}  2^{-d(\beta) \left(\frac{2 + \gamma}{2} \right)\left(1 - \frac{2}{q}\right)} 2^{ d(\eta ') (2 + \gamma) (1 - \frac{2}{q})}  \nonumber \\ & \times 2^{- \frac{2 + \gamma}{q + \epsilon_2} (d(\beta) - d(\eta'))} \left(\frac{|f|^{ -q - \epsilon_2} (\widetilde{S_{\widetilde{\eta '}}} )} {A_\gamma( \widetilde{S_{\widetilde{\eta '}}})}  \right) ^\frac{1}{q + \epsilon_2}.      \nonumber   \end{align} Combining these two inequalities gives us that \begin{align}     \{ B_\gamma  (|f k_u ^{ 1- \frac{2}{p }}|^{p + \epsilon_1}) & (u)  \}^\frac{1}{p + \epsilon_1}   \{B_\gamma (|f^{-1} k_u ^{ 1- \frac{2}{q }}|^{q + \epsilon_2})  (u) \}^\frac{1}{q + \epsilon_2}  \nonumber \\ &   \lesssim \sum_{o < \eta, \eta'  \leq \beta}   2^{ d(\eta) (2 + \gamma) \left(1 - \frac{2}{p}\right)} 2^{ d(\eta ' ) (2 + \gamma) \left(1 - \frac{2}{q}\right)}  2 ^{- \frac{2 + \gamma}{p + \epsilon_1} (d(\beta) - d(\eta)) } \nonumber \\ &   \times 2^{- \frac{2 + \gamma}{q + \epsilon_2} (d(\beta) - d(\eta '))} \left(\frac{|f|^{ p + \epsilon_1} (\widetilde{S_{\widetilde{\eta}}})}{A_\gamma( \widetilde{S_{\widetilde{\eta}}})}   \right) ^\frac{1}{p + \epsilon_1} \left(\frac{|f|^{ -q - \epsilon_2} (\widetilde{S_{\widetilde{\eta '}}} )} {A_\gamma( \widetilde{S_{\widetilde{\eta '}}})}  \right) ^\frac{1}{q + \epsilon_2} \nonumber  \end{align}

Now observe that if $\eta, \eta' \leq \beta$, then we either have that $ \eta \leq \eta'$ or $\eta' \leq \eta$.  Thus, without loss of generality, we need to bound the following quantity by a constant that is independent of $\beta \in \mathcal{D}$ :\begin{align} \sum_{o < \eta \leq \eta'  \leq \beta}   & 2^{ d(\eta) (2 + \gamma) \left(1 - \frac{2}{p}\right)} 2^{ d(\eta ' ) (2 + \gamma) \left(1 - \frac{2}{q}\right)}   2 ^{- \frac{2 + \gamma}{p + \epsilon_1} (d(\beta) - d(\eta)) } \nonumber \\  & \times 2^{- \frac{2 + \gamma}{q + \epsilon_2} (d(\beta) - d(\eta '))} \left(\frac{|f|^{ p + \epsilon_1} (\widetilde{S_{\widetilde{\eta}}})}{A_\gamma( \widetilde{S_{\widetilde{\eta}}})}   \right) ^\frac{1}{p + \epsilon_1} \left(\frac{|f|^{ -q - \epsilon_2} (\widetilde{S_{\widetilde{\eta '}}} )} {A_\gamma( \widetilde{S_{\widetilde{\eta '}}})}  \right) ^\frac{1}{q + \epsilon_2} \tag{5.12} \end{align} and we need to do the same when the above sum is taken over $\{ \eta, \eta' \in \mathcal{D}  : o < \eta' \leq \eta  \leq \beta \}$.

\medskip We first estimate $(5.12)$ for $\eta \leq \eta'  \leq \beta$.   Note that that \begin{align} \frac{1}{A_\gamma(\widetilde{S_{\widetilde{\eta'}}})}  \approx 2^{(d(\eta')  - d(\eta))(2 + \gamma)}  \frac{1}{A_\gamma(\widetilde{S_{\widetilde{\eta}}})}. \tag{5.13} \end{align}  Moreover, since the conclusion of Lemma $5.5$ holds when $\widetilde{S_{\widetilde{\eta}}}$ replaces $S_{\widetilde{\eta}}$ for any $\eta \in \mathcal{D}$, it is not difficult to check that the conclusion of Lemma $5.9$ holds when $\widetilde{S_{\widetilde{\eta}}}$ replaces $S_{\widetilde{\eta}}$.  Thus, since $\widetilde{S_{\widetilde{\eta'}}} \subseteq \widetilde{S_{\widetilde{\eta}}}$, we have that \begin{align} \int_{\widetilde{S_{\widetilde{\eta'}}}} | f| ^{-q - \epsilon_2} \, dA_\gamma \lesssim 2^{-  (d(\eta')  - d(\eta))(2 + \gamma)\left(\frac{\epsilon_2}{q + \epsilon_2}\right) }  \int_{\widetilde{S_{\widetilde{\eta}}}} | f| ^{-q - \epsilon_2} \, dA_\gamma  \tag{5.14} \end{align} for small enough $\epsilon_2$.  Also, an application of Lemma $5.5$ and Lemma $5.8$ (where again $\widetilde{S_{\widetilde{\eta}}}$ replaces $S_{\widetilde{\eta}}$) gives us that \begin{align} \left(\frac{|f|^{ p + \epsilon_1} (\widetilde{S_{\widetilde{\eta}}})}{A_\gamma( \widetilde{S_{\widetilde{\eta}}})}   \right) ^\frac{1}{p + \epsilon_1} \left(\frac{|f|^{ -q - \epsilon_2} (\widetilde{S_{\widetilde{\eta }}} )} {A_\gamma( \widetilde{S_{\widetilde{\eta }}})}  \right) ^\frac{1}{q + \epsilon_2} \leq C \tag{5.15} \end{align} where $C$ is independent of $\eta \in \mathcal{D}$.

\medskip
\noindent Plugging $(5.13), (5.14)$, and $(5.15)$ into $(5.12)$  gives us that \begin{align}  &  \sum_{o < \eta \leq \eta'    \leq \beta}    2^{ d(\eta) (2 + \gamma) \left(1 - \frac{2}{p}\right)}      2^{ d(\eta ' ) (2 + \gamma) \left(1 - \frac{2}{q}\right)}    2 ^{- \frac{2 + \gamma}{p + \epsilon_1} (d(\beta) - d(\eta)) } \nonumber \\    & \qquad \quad \times  2^{- \frac{2 + \gamma}{q + \epsilon_2} (d(\beta) - d(\eta '))} \left(\frac{|f|^{ p + \epsilon_1} (\widetilde{S_{\widetilde{\eta}}})}{A_\gamma( \widetilde{S_{\widetilde{\eta}}})}   \right) ^\frac{1}{p + \epsilon_1} \left(\frac{|f|^{ -q - \epsilon_2} (\widetilde{S_{\widetilde{\eta '}}} )} {A_\gamma( \widetilde{S_{\widetilde{\eta '}}})}  \right) ^\frac{1}{q + \epsilon_2}   \nonumber \\ & \lesssim \sum_{ \eta
 \leq \eta'  \leq \beta}    2^{ d(\eta) (2 + \gamma) \left(1 - \frac{2}{p}\right)} 2^{ d(\eta ' ) (2 + \gamma) \left(1 - \frac{2}{q}\right)}     2^{ \frac{2 + \gamma}{q + \epsilon_2} (d(\eta')  - d(\eta))} & \nonumber \\    &  \qquad \quad \times  2^{- \frac{\epsilon_2}{(q + \epsilon_2)^2} (d(\eta')  - d(\eta))(2 + \gamma) } 2 ^{- \frac{2 + \gamma}{p + \epsilon_1} (d(\beta) - d(\eta)) }   2^{- \frac{2 + \gamma}{q + \epsilon_2} (d(\beta) - d(\eta '))}  \nonumber  \\ & = \sum_{\eta'  \leq \beta}    2 ^{-  (d(\beta) - d(\eta')) (2 + \gamma) \left(\frac{1}{p + \epsilon_1} + \frac{1}{q + \epsilon_2} \right)}  \nonumber  \\ &   \qquad \quad \times  \sum_{\eta \leq \eta'}    2^{- (d(\eta')  - d(\eta))(2 + \gamma)\left(\frac{\epsilon_2}{(q + \epsilon_2)^2}    -  \frac{\epsilon_1}{p(p + \epsilon_1)} +             \frac{\epsilon_2}{q(q + \epsilon_2)}   \right)}  \tag{5.16} \end{align}

 \noindent
 Similarly, we have that \begin{align}   & \sum_{o < \eta' \leq \eta  \leq \beta} 2^{ d(\eta) (2 + \gamma) \left(1 - \frac{2}{p}\right)} 2^{ d(\eta ' ) (2 + \gamma) \left(1 - \frac{2}{q}\right)}   2 ^{- \frac{2 + \gamma}{p + \epsilon_1} (d(\beta) - d(\eta)) } \nonumber \\  & \qquad \quad  \times 2^{- \frac{2 + \gamma}{q + \epsilon_2} (d(\beta) - d(\eta '))} \left(\frac{|f|^{ p + \epsilon_1} (\widetilde{S_{\widetilde{\eta}}})}{A_\gamma( \widetilde{S_{\widetilde{\eta}}})}   \right) ^\frac{1}{p + \epsilon_1} \left(\frac{|f|^{ -q - \epsilon_2} (\widetilde{S_{\widetilde{\eta '}}} )} {A_\gamma( \widetilde{S_{\widetilde{\eta '}}})}  \right) ^\frac{1}{q + \epsilon_2} \nonumber \\ & \lesssim  \sum_{ \eta  \leq \beta}    2 ^{-  (d(\beta) - d(\eta)) (2 + \gamma) \left(\frac{1}{p + \epsilon_1} + \frac{1}{q + \epsilon_2} \right)} \nonumber  \\ & \qquad \quad  \times \sum_{\eta' \leq \eta}    2^{- (d(\eta)  - d(\eta'))(2 + \gamma)\left(\frac{\epsilon_1}{(p + \epsilon_1)^2}     - \frac{\epsilon_2}{q(q + \epsilon_2)}   +             \frac{\epsilon_1}{p(p + \epsilon_1)}         \right)} \tag{5.17} \end{align}

\noindent
Clearly the sums $(5.16)$ and $(5.17)$ converge to a sum that has an upper bound independent of $\beta \in \mathcal{D}$ if we simultaneously have \begin{displaymath} \left\{\begin{array}{ll} \frac{\epsilon_2}{(q + \epsilon_2)^2}   +  \frac{\epsilon_2}{q(q + \epsilon_2)}  >     \frac{\epsilon_1}{p(p + \epsilon_1)}        \\   \frac{\epsilon_1}{(p + \epsilon_1)^2}        +             \frac{\epsilon_1}{p(p + \epsilon_1)} > \frac{\epsilon_2}{q(q + \epsilon_2)}
\end{array} \right. .\end{displaymath}
\noindent Moreover, both of these are trivially satisfied if $\frac{\epsilon_2}{q(q + \epsilon_2)}  =     \frac{\epsilon_1}{p(p + \epsilon_1)}$ or $\epsilon_1 = \frac{\epsilon_2 p^2 }{ q^2 + \epsilon_2 q - \epsilon_2 p}$ and so the proof is complete so long as $\epsilon_2 > 0$ is set small enough.

\bigskip
\bigskip
Finally in this paper, we will prove Propositions $4.1$ and $4.4$, starting with Proposition $4.4$.  The proof is similar to the proof of Theorem $3.2.2$ in \cite{H}, though we include it since some of the details are different.  Let $d(u, v)$ denote the non-isotropic metric on $\partial \mathbb{B}_n$ given by $d(u, v) = |1 - u \cdot v|^\frac{1}{2}$ and let $B = B(u, r)$ denote a ball in this metric. It is well known (see \cite{Z2}) that $B(u, r) = \partial \mathbb{B}_n $ when $r \geq \sqrt{2}$ and that there exists $C > 0$ independent of $r$ and $u $ such that \begin{align} \frac{1}{C} r^{2n} \leq \sigma(B(u, r)) \leq C r^{2n} \tag{5.18} \end{align} where  $\sigma$ is the canonical surface measure on $\partial \mathbb{B}_n$.

Fix some large $M > 0$ such that $C^2 M^{-n} \leq \frac{1}{2}$ where $C$ is the constant in $(5.18)$.  Without loss of generality fix some $z \in \mathbb{B}_n$ where $1/M < 1 -  |z| < 1$ and pick $J \in \mathbb{N}$ such that $M^{-J - 1} \leq 1 - |z| <  M^{-J },$ and let $B_k = B(z/|z|, M^{\frac{(k-J)}{2} })$ for $k \in \{0, 1, \ldots, J + 1  \}$.   Now, for any $0 \leq t \leq 1, \ 0 \leq a \leq 1$, and $ \theta  \in \mathbb{R}$,
we have \begin{align} |1 - tae^{i\theta}|^2 = t |1 - ae^{i\theta}|^2 + (1 - t)(1 - ta^2) \geq t |1 - ae^{i\theta}|^2.  \nonumber \end{align}  Thus, if $\zeta \in \partial \mathbb{B}_n \backslash B_k$, then writing $\zeta \cdot z = tae^{i\theta}$ where $t = |z|$ and $a e^{i \theta} =   \zeta \cdot (z/|z|)$ gives us that \begin{align} |k_z (\zeta)|  = \frac{(1 - |z|^2)^{n/2}} {|1 - \zeta \cdot z |^n} \lesssim M^{-\frac{nJ}{2}} M^{-n (k - J)}  \lesssim \frac{M^{-\frac{nk}{2}}} {(\sigma(B_{k + 1}))^\frac{1}{2}} \nonumber \end{align}  Also, if $\zeta \in B_0$, then we have that $|k_z (\zeta)|  \approx (\sigma(B_0))^{-\frac{1}{2}}$

\medskip

Thus, if we define $B_{-1} = \emptyset$ then we have that \begin{align}  \left(\int_{\partial \mathbb{B}_n}  \ w |k_z|^p \, d\sigma \right)^\frac1p &  \left(\int_{\partial \mathbb{B}_n} \ w^{- \frac{1}{p - 1}} |k_z|^q \, d\sigma \right)^{\frac1q}  \nonumber \\ & \leq \sum_{k, k' = -1}^{J } \left(\int_{B_{k + 1} \backslash B_k} \ w |k_z|^p \, d\sigma \right) ^\frac1p \left(\int_{B_{k' + 1} \backslash B_{k'} } \ w^{- \frac{1}{p - 1}} |k_z|^q \, d\sigma \right)^{\frac1q} \nonumber  \\ & \lesssim  \sum_{k, k' = -1}^{J } \frac{M^{-\frac{nk
}{2}}}  {(\sigma(B_{k + 1}))^{\frac{1}{2}}} \frac{M^{-\frac{nk'}{2}}}  {(\sigma(B_{k' + 1}))^{\frac{1}{2}}}   \left(\int_{B_{k + 1} }  w \, d\sigma \right)^\frac{1}{p} \left(\int_{B_{k' + 1}  }  w^{- \frac{1}{p - 1}} \, d\sigma \right)^{\frac1q}   \tag{5.19}   \end{align}
\bigskip

Now break the sum in $(5.19)$ into two sums, the first of which is taken over $k \leq k'$ and the second over $k' < k$.
In the first case, we have that \begin{align}  \frac{M^{-\frac{nk}{2}}}  {(\sigma(B_{k + 1} ) )^{\frac{1}{2}}}  \lesssim \frac{M^{-\frac{nk}{2}} M^{  \frac{ n (k' - k)}{2} } } {(\sigma(B_{k' + 1}))^{\frac{1}{2}}} = \frac{M^{{-nk}} M^{  \frac{ nk' }{2} } } {(\sigma(B_{k' + 1}))^{\frac{1}{2}} }. \nonumber \end{align}  Moreover, similar to Lemma $5.6$, we have that $\frac{w(B_{k+ 1}) }{w(B_{k' + 1})} \leq \delta_1 ^{k' - k} \nonumber $ where $\delta_1 = 1 - (2^p \|w\|_{\text{A}_p}  )^{-1} $.  Thus, we have that  \begin{align} \sum_{k \leq k'}^{J } & \frac{M^{-\frac{nk
}{2}}}    {(\sigma(B_{k + 1}))^{\frac{1}{2}}}  \frac{M^{-\frac{nk'}{2}}}  {(\sigma(B_{k' + 1}))^{\frac{1}{2}}}   \left(\int_{B_{k + 1} }  w \, d\sigma \right)^\frac1p  \left(\int_{B_{k' + 1}  }  w^{- \frac{1}{p - 1}} \, d\sigma \right)^{\frac1q}
\nonumber \\ & \lesssim  \sum_{k \leq k'}^{J } \frac{M^{-nk} \delta_1 ^{k' - k} }   {(\sigma(B_{k' + 1}))}    \left(\int_{B_{k' + 1} }  w \, d\sigma \right)^\frac1p  \left(\int_{B_{k' + 1}  }  w^{- \frac{1}{p - 1}} \, d\sigma \right)^\frac1q \nonumber \\ & = \sum_{k = -1}^{J } M^{-nk} \sum_{k' = k}^{J } \delta_1 ^\frac{k' - k}{p} \left(\frac{1}{(\sigma(B_{k' + 1}))}\int_{B_{k' + 1} }  w \, d\sigma \right)^\frac1p  \left(\frac{1}{(\sigma(B_{k' + 1}))} \int_{B_{k' + 1}  }  w^{- \frac{1}{p - 1}} \, d\sigma \right)^\frac1q \nonumber \\ & \lesssim   \| w\|_{\text{A}_p} ^{1 + \frac1p} \nonumber\end{align}

\noindent Similarly, for $k ' < k$ we have that $\frac{w^{- \
\frac{1}{p-1}} (B_{k' + 1}) }{w^{- \frac{1}{p-1}} (B_{k+ 1})} \leq \delta_2 ^{k - k'} $ where $\delta_2 = 1 - (2^q \|w\|_{\text{A}_p} ^\frac{q}{p} )^{-1}$, so that \begin{equation*} \sum_{k' < k}^{J }  \frac{M^{-\frac{nk
}{2}}}    {(\sigma(B_{k + 1}))^{\frac{1}{2}}}  \frac{M^{-\frac{nk'}{2}}}  {(\sigma(B_{k' + 1}))^{\frac{1}{2}}}    \left(\int_{B_{k + 1} }  w \, d\sigma \right)^\frac{1}{p}  \left(\int_{B_{k' + 1}  }  w^{- \frac{1}{p - 1}} \, d\sigma \right)^\frac1q
  \lesssim  \|w\|_{\text{A}_p}  ^{\frac{1 + q}{p}} \nonumber \end{equation*} which proves Proposition $4.4$.

\bigskip

Now to prove proposition $4.1$, let $d$ be the pseudo-metric $d(z, u) = ||z| - |u|| + |1 - \frac{z}{|z|} \cdot \frac{u}{|u|}|$ on $\mathbb{B}_n$.
According to Lemma $2$ in \cite{Be}, there exists $C > 0$ such that \begin{align} \frac{1}{C} r^{n + 1 + \gamma} \leq v_\gamma(B(u, r)) \leq C r^{n + 1 + \gamma} \nonumber \end{align} whenever $r \geq 1 - |u|$ (and $u \in \mathbb{B}_n$.)   As before, pick some large $M > 0$ where $C^2 M^{-(n + 1 + \gamma)} \leq \frac{1}{2} $ and for some fixed $\frac{1}{M} < 1 - |z| < 1$, pick $J$ where $M^{-J - 1} \leq 1 - |z| <  M^{-J }$ and let $B_k = B(z,  M^{k-J })$ for $k \in \{0, 1, \ldots, J + 1 \}$.  Note that we clearly have $z/|z| \in B_k$ for each $k$ and note that $M^{k - J } \geq 1 - |z|$.
  Furthermore, if we can show that \begin{equation} \sup_{u \in \mathbb{B}_n \backslash B_k} |k_z(u)| \lesssim \frac{M^{- \frac{k}{2} (n + 1 + \gamma)}}{(v_\gamma(B_{k + 1}))^\frac12} \end{equation} then the proof of Proposition $4.1$ will be almost identical to the proof of Proposition $4.4$ above.

  To that end, set $t = |z||u|$ and $ae^{i \theta} = \frac{z}{|z|} \cdot \frac{u}{|u|}$ so $0 \leq a, t < 1$.  Then as before we can write \begin{align} |1 - z \cdot u|^2 & =  t|1 - ae^{i\theta}|  + (1 - t)(1 - ta^2) \nonumber \\ & \geq |z||u| \left|1 - \frac{z}{|z|} \cdot \frac{u}{|u|}\right|^2 + (1 - |z||u|)^\frac12. \end{align} First note that we can obviously assume that $|u| \geq \frac12$ since otherwise (5.1) is obviously true.  Now if $u \not \in B_k$ then either $||z| - |u|| \geq \frac12 M^{k - J}$ or $\left|1 - \frac{z}{|z|} \cdot \frac{u}{|u|}\right| \geq \frac12 M^{k - J}$.  In the latter case we clearly have $|1 - z \cdot u| \gtrsim M^{k - J}$ and the former case we have \begin{equation*} 1 - |u| \geq 1 - |z| + \frac12 M^{k - J} \geq M^{- J - 1} + \frac12 M^{k - J}  \gtrsim M ^{k - J} \end{equation*} so again $|1 - z \cdot u| \gtrsim M^{k - J}$.  Thus, \begin{equation*}  \sup_{u \in \mathbb{B}_n \backslash B_k} |k_z(u)| \lesssim \frac{M^{- \frac{J}{2} (n + 1 + \gamma)}}{M^{(k - J) (n + 1 + \gamma)}} \approx \frac{M^{- \frac{k}{2} (n + 1 + \gamma)}}{(v_\gamma(B_{k + 1}))^\frac12}. \end{equation*}

\end{document}